\documentclass[final]{siamltex}

\usepackage{graphics,epsfig}
\usepackage{enumerate}
\usepackage{amsmath}
\usepackage{amssymb,amscd,mathrsfs}
\usepackage{pifont}

\allowdisplaybreaks

\newtheorem{assumption}		[theorem]{Assumption}
\newtheorem{claim}				[theorem]{Claim}

\newcommand{\ba}				{\begin{array}}
\newcommand{\ea}				{\end{array}}
\newcommand{\er}[1]			{{(\ref{#1})}}
\newcommand{\nn}				{{\nonumber}}

\newcommand{\ooer}[1]			{\normalfont\tag{\ref{#1}}}

\newcommand{\bo}				{{\mathcal{L}}}

\newcommand{\cN}				{{\mathscr{N}}}
\newcommand{\cNcom}			{{{\cN}^{\perp}}}
\newcommand{\cR}				{{\mathscr{R}}}
\newcommand{\cS}				{{\mathscr{S}}}
\newcommand{\cW}				{{\mathscr{W}}}
\newcommand{\cX}				{{\mathscr{X}}}
\newcommand{\cY}				{{\mathscr{Y}}}

\newcommand{\C}				{{\mathbb{C}}}

\newcommand{\ddtone}[2]		{{\frac{d {#1}}{d {#2}}}}
\newcommand{\demi}			{{\textstyle\frac{1}{2}}}
\newcommand{\dom}			{{\textsf{dom}}}
\newcommand{\eps}				{{\epsilon}}
\newcommand{\Frechet}			{{Fr\'{e}chet}}

\newcommand{\fundterm}			{{\widetilde\psi}}

\newcommand{\Ltwo}			{{\mathcal{L}_{2}}}

\newcommand{\mapsinto}		{\rightarrow}

\newcommand{\N}				{{\mathbb{N}}}
\newcommand{\ol}[1]			{{\overline{#1}}}
\newcommand{\op}[1]			{{\mathcal{#1}}}
\newcommand{\opdot}[1]			{{\dot{\op{#1}}}}
\newcommand{\ophat}[1]			{{\widehat{\op{{#1}}}}}
\newcommand{\opMtilde}			{{\widetilde{\op{M}}}}
\newcommand{\opPtilde}			{{\widetilde{\op{P}}}}
\newcommand{\opsqrt}[1]			{{\widehat{\op{{#1}}}}}
\newcommand{\optilde}[1]		{{\widetilde{\op{{#1}}}}}

\newcommand{\pdtone}[2]		{{\frac{\partial {#1}}{\partial {#2}}}}
\newcommand{\R}				{{\mathbb{R}}}
\newcommand{\Z}				{{\mathbb{Z}}}
\newcommand{\semiconcave}[1]	{{{\mathscr{S}}_{-}^{{#1}}}}
\newcommand{\semiconvex}[1]	{{{\mathscr{S}}^{{#1}}}}
\newcommand{\scf}[2]			{{\cS^{{#1}}\left({#2}\right)}}
\newcommand{\skx}				{{\scf{\op{K}}{\cX}}}

\newcommand{\ts}[1]			{{\textstyle{#1}}}

\newcommand{\opBmp}[1]		{{\op{B}_{{#1}}^{\oplus}}}

\newcommand{\opNtilde}		{{\widetilde{\op{N}}}}
\newcommand{\betaOOhat}[1]	{{\widehat{\op{B}}_{{#1}}^{1,1}}}
\newcommand{\betaOThat}[1]	{{\widehat{\op{B}}_{{#1}}^{1,2}}}
\newcommand{\betaTThat}[1]	{{\widehat{\op{B}}_{{#1}}^{2,2}}}
\newcommand{\betaOO}[1]	{{{\op{B}}_{{#1}}^{1,1}}}
\newcommand{\betaOT}[1]	{{{\op{B}}_{{#1}}^{1,2}}}
\newcommand{\betaTT}[1]	{{{\op{B}}_{{#1}}^{2,2}}}

\newcommand{\be}{\begin{equation}}
\newcommand{\ee}{\end{equation}}
\newcommand{\beasnum}{\begin{eqnarray}}
\newcommand{\eeasnum}{\end{eqnarray}}
\newcommand{\beas}{\begin{eqnarray*}}
\newcommand{\eeas}{\end{eqnarray*}}

\newtheorem{sideremark}         [theorem]{Remark}
\newenvironment{remark}         {\begin{sideremark}\rm}{\end{sideremark}}
\title{
A max-plus dual space fundamental solution for a class of operator differential Riccati equations\thanks{
This research is supported by grants 
from the Australian Research Council, AFOSR, and NSF. Preliminary results contributing to this paper appear in \cite{DM:11,DM:12,DM:14}.
}}

\author{Peter M. Dower\thanks{Department of Electrical \& Electronic Engineering, University of Melbourne, Victoria 3010, Australia.
{\tt pdower@unimelb.edu.au}, +61-3-8344-6711.}
\and William M. McEneaney\thanks{Department of Mechanical \& Aerospace Engineering, University of California at San Diego, 9500 Gilman Dr., La Jolla, CA 92093-0411, USA. {\tt wmceneaney@eng.ucsd.edu}, +1-858-822-5835.}
}

\begin{document}

\maketitle

\begin{abstract}
A new fundamental solution semigroup for operator differential Riccati equations is developed. This fundamental solution semigroup is constructed via an auxiliary finite horizon optimal control problem whose value functional growth with respect to time horizon is determined by a particular solution of the operator differential Riccati equation of interest. By exploiting semiconvexity of this value functional, and the attendant max-plus linearity and semigroup properties of the associated dynamic programming evolution operator, a semigroup of max-plus integral operators is constructed in a dual space defined via the Legendre-Fenchel transform. It is demonstrated that this semigroup of max-plus integral operators can be used to propagate all solutions of the operator differential Riccati equation that are initialized from a specified class of initial conditions. As this semigroup of max-plus integral operators can be identified with a semigroup of quadratic kernels, an explicit recipe for the aforementioned solution propagation is also rendered possible.
\end{abstract}




\section{Introduction}
The objective of this paper is to develop a new fundamental solution semigroup for operator differential Riccati equations of the form
\begin{align}
	\dot{\op{P}}(t)
	& =  \op{P}(t)\, \op{A} + \op{A}' \, \op{P}(t)+ \op{P}(t) \, \sigma \, \sigma' \, \op{P}(t) + \op{C}\,,
	\label{eq:op-P}
\end{align}
where $\op{P}(t)$ is a self-adjoint bounded linear operator evolved to time $t$ from some initialization $\op{P}(0)$, $\op{A}$ is an unbounded,  densely defined and boundedly invertible linear operator that generates a $C_0$-semigroup of bounded linear operators, and $\sigma$ and $\op{C}$ are bounded linear operators, all defined with respect to a pair of underlying Hilbert spaces $\cX$ and $\cW$. Operator differential Riccati equations of this form arise naturally in the formulation and solution of optimal control problems for linear infinite dimensional systems \cite{BDDM:07,CZ:95}. Their solution is of particular interest where a state feedback characterization for an optimal control is sought. The fundamental solution semigroup obtained generalizes the finite dimensional case presented in \cite{M:08}, and the specific infinite dimensional cases documented in \cite{DM:11,DM:12} for mild solutions. It describes all solutions of \er{eq:op-P} corresponding to a class of quadratic terminal payoffs. Preliminary results in this direction also appear in \cite{DM:14}.

Development of the new fundamental solution semigroup for \er{eq:op-P} proceeds by considering an infinite dimensional optimal control problem on a finite time horizon $t$. This control problem is constructed such that the value functional obtained exhibits quadratic growth with respect to the state variable, where the growth is determined by the solution $\op{P}(t)$ of the operator differential Riccati equation \er{eq:op-P} at time $t$. Consequently, evolution of the solution $\op{P}(t)$ of \er{eq:op-P} with time $t$ can be identified with evolution of the value functional with respect to time horizon $t$, with dynamic programming \cite{BCD:97,B:57} providing a mechanism for the latter. As the value functional obtained is demonstrably semiconvex, its evolution via dynamic programming can be identified with a corresponding evolution in a dual space defined via the Legendre-Fenchel transform \cite{R:74}. Critically, by exploiting max-plus linearity of the dynamic programming evolution operator, this dual space evolution can be decoupled from the terminal payoff employed in the optimal control problem, and hence from the initial data that defines any specific solution $\op{P}$ of \er{eq:op-P}. Indeed, the set of time horizon indexed dual space evolution operators defined via this decoupling describes a {\em fundamental solution} to \er{eq:op-P}, as its elements can be used to propagate any initial data $\op{P}(0)$ within a specific class of operators to yield the corresponding solution $\op{P}(t)$ of \er{eq:op-P} at time $t\in\R_{\ge 0}$. As dynamic programming naturally endows the value functional with a semigroup property, this set of time horizon indexed dual space evolution operators also defines a semigroup that can be regarded as the {\em max-plus dual space fundamental solution semigroup} for \er{eq:op-P}. 

In terms of organization, the operator differential Riccati equation of interest is posed in Section \ref{sec:Riccati}, along with results concerning existence and uniqueness of its solution on a finite time horizon. Construction of the max-plus fundamental solution semigroup is presented in detail in Section \ref{sec:fund}, with the steps involved in applying this fundamental solution semigroup to evaluate solutions of \er{eq:op-P} enumerated in Section \ref{sec:solve}. This is followed by some brief conclusions in Section \ref{sec:conc}, and appendices that include (for completeness) pertinent well-known details concerning continuity of operator-valued functions, the Yosida approximation, and so on.




\section{Operator differential Riccati equation}
\label{sec:Riccati}
Attention is initially restricted to the operator differential Riccati equation \er{eq:op-P} of interest, with sufficient conditions for existence and uniqueness of solutions on a finite horizon established following the approach of \cite{BDDM:07}. Two auxiliary operator differential equations of subsequent utility are similarly considered. In generalizing the finite dimensional Riccati equations considered in \cite{M:08}, the main technical challenges in the infinite dimensional setting considered here concern the notion of solution for operator differential equations, the consequences of an unbounded $\op{A}$ in analyzing those equations and solutions, and an appropriate notion of semiconvexity for functionals on infinite dimensional spaces. Tools for dealing with these challenges are well-understood, and are cited directly from \cite{BDDM:07,CZ:95,P:83} and \cite{BN:00,HRS:01,R:74} as required. Otherwise, the development largely follows that of the finite dimensional case \cite{M:08}.

\subsection{Riccati equation}
Consider the operator differential Riccati equation posed with respect to Hilbert spaces $\cX$ and $\cW$ by
\begin{align}
	\opdot{P}(t)
	& = \op{A}' \, \op{P}(t) + \op{P}(t)\, \op{A} + \op{P}(t)\, \sigma\, \sigma'\, \op{P}(t) + \op{C}\,,
	\ooer{eq:op-P}
\end{align}
in which $\op{A}:\dom(\op{A})\subset\cX\mapsinto\cX$ is unbounded and densely defined on $\cX$, $\sigma\in\bo(\cW;\cX)$, $\op{C}\in\bo(\cX)$ is self-adjoint and non-negative, $\op{A}'$ and $\sigma'$ denote the respective adjoints of $\op{A}$ and $\sigma$, and $t\in[0,\tau^*]$ for some $\tau^*\in\R_{>0}$. (Throughout, $\bo(\cX;\cY)$ is used to denote the space of bounded linear operators mapping from Banach space $\cX$ to Banach space $\cY$. Where $\cX = \cY$, this notation is abbreviated to $\bo(\cX)$.)

\begin{assumption}
\label{ass:op-A}
$\op{A}$ is boundedly invertible and generates a $C_0$-semigroup of bounded linear operators. 
\end{assumption}

In order to define solutions for the operator differential Riccati equation \er{eq:op-P}, it is convenient \cite{BDDM:07} to 
define two sets of self-adjoint bounded linear operators by
\begin{align}
	\Sigma(\cX)
	& \doteq \left\{ \op{P}\in\bo(\cX) \, \biggl| \, \op{P} \text{ is self-adjoint} \right\},
	\label{eq:Sigma}
	\\
	\Sigma_{\op{M}}(\cX)
	& \doteq \left\{ \op{P}\in\Sigma(\cX) \, \left| \ba{c}
									 \op{P} - \op{M} \text{ is coercive} 
									 \\
									 \text{on } \dom(\op{A})
								\ea \right. \right\},
	\quad\op{M}\in\Sigma(\cX)\,.
	\label{eq:Sigma-bar}
\end{align}
An operator $\op{P}:\dom(\op{P})\subset\cX\mapsinto\cX$ is coercive if there exists an $\eps\in\R_{>0}$ such that $\langle x,\, \op{P}\, x \rangle \ge \eps\, \|x\|^2$ for all $x\in\dom(\op{P})$. If $\op{P}$ is both coercive and self-adjoint, then it has a bounded inverse, see \cite[Example A.4.2, p.609]{CZ:95} and \cite[Problem 10, p.535]{K:78}.

Relevant spaces of (uniformly) continuous and strongly continuous operator-valued functions defined on an interval $I \doteq [0,T]\subset\R_{\ge 0}$, $T\in\R_{>0}$, and taking values in $\Sigma\in\{\bo(\cX),\, \Sigma(\cX),\, \Sigma_{\op{M}}(\cX)\}$, are 
\begin{align}
	C(I;\Sigma)
	& \doteq \left\{ \op{F}:I\mapsinto\Sigma \, \left| \ba{c} 
								\op{F} \text{ is}
								\\
								\text{continuous}
						\ea
						\right. 
	\right\},
	\label{eq:space-C}
	\\
	C_0(I;\Sigma)
	& \doteq \left\{ \op{F}:I\mapsinto\Sigma \, \left| \ba{c} 
								\op{F} \text{ is strongly}
								\\
								\text{continuous}
						\ea \right. \right\}\,.
	\label{eq:space-C0}
\end{align}
\begin{remark}
Any bounded linear operator is the generator of a uniformly continuous semigroup of bounded linear operators (see \cite[Theorem 1.2, p.2]{P:83}). In constrast, the unbounded and densely defined operator $\op{A}:\dom(\op{A})\subset\cX\mapsinto\cX$ of \er{eq:op-P} is the generator of a strongly ($C_0$-) semigroup of bounded linear operators by Assumption \ref{ass:op-A} (which also implies that $\op{A}$ is closed, see also \cite[Corollary 2.5, p.5]{P:83}). These semigroups are subsets of $C(I;\bo(\cX))$ and $C_0(I;\bo(\cX))$ respectively. 
Elements of such a semigroup, denoted by $e^{\op{A}\, t}$ for $t\in\R_{\ge 0}$, satisfy the usual semigroup properties, with $e^{\op{A}\, 0} = \op{I}$ (the identity operator), and $e^{\op{A}\, t} \, e^{\op{A}\, s} = e^{\op{A}\,(t+s)} = e^{\op{A}\,s}\, e^{\op{A}\, t}$ for all $s,t\in\R_{\ge 0}$ such that $s,t,s+t\in I$. (See also Remark \ref{rem:semigroup}.) Note finally that $C(I;\Sigma)\subset C_0(I;\Sigma)$ for $\Sigma\in\{\bo(\cX),\, \Sigma(\cX),\, \Sigma_{\op{M}}(\cX)\}$, where $C(I;\Sigma(\cX)) \subset C(I;\bo(\cX))$ and $C_0(I;\Sigma(\cX))\subset C_0(I;\bo(\cX))$ define vector spaces, while $C(I;\Sigma_\op{M}(\cX))$ and $C_0(I;\Sigma_\op{M}(\cX))$ are merely subsets of those vector spaces. (See Appendix \ref{app:spaces}.)
\end{remark}

A {\em mild} solution of the operator differential Riccati equation \er{eq:op-P} on a time interval $[0,T]$, $T\in\R_{>0}$, is any operator-valued function $\op{P}\in C_0([0,T]; \Sigma(\cX))$ that satisfies 
\begin{align}
	\op{P}(t) \, x
	& = \gamma(\op{P})(t)\, x\,,
	\label{eq:op-P-mild}
\end{align}
for all $x\in\cX$, $t\in[0,T]$, with $\gamma(\op{O})$ defined for every $\op{O}\in C_0([0,T];\Sigma(\cX))$ by
\begin{equation}
	\begin{aligned}
	& \hspace{-2mm}
	\left[ \gamma(\op{O}) (t)\right] x
	\doteq e^{\op{A}'\, t} \, \op{O}(0) \, e^{\op{A}\, t}\, x
	+ \int_0^t e^{\op{A}'\, (t-s)} \left[ \op{O}(s) \, \sigma\, \sigma' \, \op{O}(s) + \op{C} \right] e^{\op{A}\, (t-s)}\, x \, ds
	\end{aligned}
	\label{eq:op-gamma}
\end{equation}
for all $t\in[0,T]$, $x\in\cX$, where $e^{\op{A}'\, \cdot}$ denotes the $C_0$-semigroup generated by the operator adjoint $\op{A}'$ (see, for example,  \cite[Theorem 2.2.6, p.37]{CZ:95}). As per \cite{BDDM:07}, it is convenient to introduce an analogous operator differential Riccati equation to \er{eq:op-P}, defined with respect to the {\em Yosida approximations} $\op{A}_n\in\bo(\cX)$ of $\op{A}$ defined for all $n\in\N$, see Appendix \ref{app:Yosida}. In particular,
\begin{align}
	\opdot{P}_n(t)
	& = \op{A}_n' \, \op{P}_n(t) + \op{P}_n(t)\, \op{A}_n + \op{P}_n(t)\, \sigma\, \sigma'\, \op{P}_n(t) + \op{C}\,.
	\label{eq:op-P-n}
\end{align}
Similarly, a solution of \er{eq:op-P-n} on a time interval $[0,T]$, $T\in\R_{>0}$, is any operator-valued function $\op{P}_n \in C([0,T];\Sigma(\cX))$ that satisfies 
\begin{align}
	\op{P}_n(t) \, x
	& = \gamma_n(\op{P}_n)(t)\, x\,,
	\label{eq:op-P-n-mild}
\end{align}
for all $x\in\cX$, $t\in[0,T]$, with $\gamma_n(\op{O})$ defined for every $\op{O}\in C([0,T];\Sigma(\cX))$ by
\begin{equation}
	\begin{aligned}
	& \hspace{-2mm} \left[ \gamma_n(\op{O})(t) \right] x
	\doteq e^{\op{A}_n'\, t} \, \op{O}(0) \, e^{\op{A}_n\, t}\, x
	+ \int_0^t e^{\op{A}_n'\, (t-s)} \left[ \op{O}(s) \, \sigma\, \sigma' \, \op{O}(s) + \op{C} \right] e^{\op{A}_n\, (t-s)}\, x \, ds
	\end{aligned}
	\label{eq:op-gamma-n}
\end{equation}
for all $t\in[0,T]$, $x\in\cX$, where $e^{\op{A}_n'\, t}\in\bo(\cX)$ denotes an element of the uniformly continuous semigroup generated by the Yosida approximation adjoint $\op{A}_n'\in\bo(\cX)$.

\begin{remark}
Equations \er{eq:op-P-mild} and \er{eq:op-P-n-mild} are integral equations that may be derived formally from their operator differential Riccati equation counterparts \er{eq:op-P} and \er{eq:op-P-n}, respectively, see Appendix \ref{app:integral-forms}.
\end{remark}

\begin{theorem}
\label{thm:existence-uniqueness-P}
Given any $\op{P}_0\in\Sigma(\cX)$, there exists a $\tau\in\R_{>0}$ such that the operator differential Riccati equations \er{eq:op-P}, \er{eq:op-P-n} exhibit respective unique solutions $\op{P}\in C_0([0,\tau];\Sigma(\cX))$, $\op{P}_n\in C([0,\tau];\Sigma(\cX))$ satisfying $\op{P}(0) = \op{P}_0 = \op{P}_n(0)$ for all $n\in\N$. Furthermore, for all $x\in\cX$, 
\begin{align}
	\lim_{n\rightarrow\infty} \op{P}_n(\cdot) \, x = \op{P}(\cdot)\, x\,,
	\label{eq:op-limit}
\end{align}
where the limit is defined with respect to the Banach space $(C([0,\tau];\cX),\, \|\cdot\|_{C([0,\tau];\cX)})$, in which $\|x\|_{C([0,\tau];\cX)} \doteq \sup_{t\in[0,\tau]} \|x(t)\|$ for all $x\in C([0,\tau];\cX)$.
\end{theorem}

\begin{remark}
The limit \er{eq:op-limit} is a statement of strong operator convergence of $\op{P}_n(\cdot)\, x$ to $\op{P}(\cdot)\, x$ in $C([0,\tau];\cX)$, given $x\in\cX$. This is strictly weaker than uniform operator convergence of $\op{P}_n$ to $\op{P}$ in $C([0,\tau];\bo(\cX))$ via the norm $\|\cdot\|_{C[0,\tau]}$ (see, for example, \cite[p.263]{K:78}). As $(C([0,\tau];\bo(\cX)),\, \|\cdot\|_{C[0,\tau]})$ defines a Banach space, strong operator convergence allows the limit \er{eq:op-limit} to reside in $C_0([0,\tau];\bo(\cX))\setminus C([0,\tau];\bo(\cX))$ when $\op{A}$ is unbounded.
\end{remark}

\begin{remark}
It may also be noted, by an analogous argument to \cite[Proposition 2.1, p.391]{BDDM:07}, that $\op{P}\in C_0([0,\tau];\bo(\cX))$ is a mild solution of \er{eq:op-P} if and only if it is a {\em weak} solution (see \cite[Definition 2.1, p.390]{BDDM:07}). 
\end{remark}

\begin{proof}[Theorem \ref{thm:existence-uniqueness-P}]
The proof follows that of \cite[Lemma 2.2, p.391]{BDDM:07}, while demonstrating global uniqueness on a finite horizon. It is not extended to the infinite horizon due to the possibility of finite escape. Fix $T\in\R_{>0}$. Given the unbounded and densely defined linear operator $\op{A}$ satisfying Assumption \ref{ass:op-A} and as per \er{eq:op-P}, the main Yosida approximation Theorem \ref{thm:Yosida} (see Appendix \ref{app:Yosida} and \cite{BDDM:07,P:83}) implies that there exists $M\in\R_{\ge 1}$ and $\omega\in\R_{\ge 0}$ such that the Yosida approximations $\op{A}_n\in\bo(\cX)$, $n\in\N$, are well-defined and satisfy \er{eq:Yosida-limit}, \er{eq:Yosida-bounds}. Hence, 
\begin{align}
	M_T 
	& \doteq \sup_{n\in\N} \sup_{t\in[0,T]} \max\left( \left\| e^{\op{A}\, t} \right\|_{\bo(\cX)},\, \left\| e^{\op{A}_n\, t} \right\|_{\bo(\cX)} \right) < \infty\,,
	\label{eq:M-T}
\end{align}
where $\|\cdot\|_{\bo(\cX)}$ denotes the induced operator norm on $\bo(\cX)$. (Note in particular that $M_T$ of \er{eq:M-T} depends only on operator $\op{A}$ and the time horizon $T\in\R_{\ge 0}$.) Fix any $\op{P}_0\in\Sigma(\cX)$, $r\in\R_{>0}$, and $\tau\in(0,T]$ such that
\begin{align}
	& \hspace{-1mm}
	r > 2\, M_T^2 \, a\,, \,
	\tau 
	< \min \left( \frac{a}{r^2  b + \|\op{C}\|_{\bo(\cX)}},\, \frac{1}{4\, r\, M_T^2\, b} \right), 
	\label{eq:tau}
\end{align} 
where $a\doteq  \|\op{P}_0\|_{\bo(\cX)}$ and $b\doteq  \| \sigma\, \sigma' \|_{\bo(\cX)}$. Let $B_{C[0,\tau]}(r)$ and $B_{C_0[0,\tau]}(r)$ denote respective balls of radius $r$ in $C([0,\tau];\Sigma(\cX))$ and $C_0([0,\tau];\Sigma(\cX))$, defined with respect to the norms $\|\cdot\|_{C[0,\tau]}$ and $\|\cdot\|_{C_0[0,\tau]}$ of \er{eq:norm-C-C0}. That is,
\begin{align}
	\begin{aligned}
	B_{C[0,\tau]}(r) & \doteq \left\{ \op{F} \in C([0,\tau];\Sigma(\cX)) \, \biggl| \, \|\op{F}\|_{C[0,\tau]} \le r \right\}\,,
	\\
	B_{C_0[0,\tau]}(r) & \doteq \left\{ \op{F} \in C_0([0,\tau];\Sigma(\cX)) \, \biggl| \, \|\op{F}\|_{C_0[0,\tau]} \le r \right\}\,.
	\end{aligned}
	\nn
\end{align}
(Note by Lemma \ref{lem:spaces} that $B_{C[0,\tau]}(r)\subset B_{C_0[0,\tau]}(r)$.) Fix any $\op{P}\in B_{C[0,\tau]}(r)$ 
satisfying $\op{P}(0) = \op{P}_0$. Applying the operator $\gamma_n$ of \er{eq:op-gamma-n} to $\op{P}$, evaluating at time $t\in[0,\tau]$, and applying the norm $\|\cdot\|_{\bo(\cX)}$ (while dropping the $\bo(\cX)$ subscript),
\begin{align}
	\left\| \gamma_n(\op{P})(t) \right\|
	& \le \|e^{\op{A}_n' t}\| \, \|\op{P}_0\|\, \|e^{\op{A}_n\, t}\| +
	\hspace{-1mm}
	\int_0^t \hspace{-2mm} \|e^{\op{A}_n' (t-s)}\|  \left[ \|\op{P}\| \, \|\sigma\, \sigma' \|\, \|\op{P}\|
			+ \|\op{C}\| \right] \|e^{\op{A}_n(t-s)}\| \, ds
	\nn\\
	& \le M_T^2 \left( a + \tau \, [ r^2\, b + \|\op{C}\| ] \right)
	\le 2\, M_T^2 \, a \le r\,,
	\nn
\end{align}
where \er{eq:M-T} and \er{eq:tau} have been applied. Taking the supremum over $t\in[0,\tau]$ (and restoring the norm subscripts),
$
	\|\gamma_n (\op{P})\|_{C([0,\tau])} 
	= \sup_{t\in[0,\tau]} \|\gamma_n(\op{P})(t)\|_{\bo(\cX)} 
	\le r
$. 
As $\op{P}\in B_{C[0,\tau]}(r)$ is arbitrary, it follows immediately that $\gamma_n:B_{C[0,\tau]}(r) \mapsinto B_{C[0,\tau]}(r)$. In order to show that $\gamma_n$ is a contraction on $B_{C[0,\tau]}(r)$, fix any $\ophat{P}\in B_{C[0,\tau]}(r)$ satisfying $\ophat{P}(0) = \op{P}_0$. Applying \er{eq:op-gamma-n} for $t\in[0,\tau]$,
\begin{align}
	\|\gamma_n(\op{P})(t) - \gamma_n(\ophat{P})(t) \|_{\bo(\cX)}
	& \le M_T^2 \int_0^t \left\| \op{P}(s)\, \sigma \, \sigma' \, \op{P}(s) - \ophat{P}(s)\, \sigma \, \sigma' \, \ophat{P}(s) \right\|_{\bo(\cX)} \, ds\,,
	\label{eq:contraction-1-a}
\end{align}
where for all $s\in[0,t]$,
\begin{align}
	& \left\| \op{P}(s)\, \sigma \, \sigma' \, \op{P}(s) - \ophat{P}(s)\, \sigma \, \sigma' \, \ophat{P}(s) \right\|_{\bo(\cX)}
	\nn\\
	& = \left\| \op{P}(s)\, \sigma \, \sigma' \, \left[ \op{P}(s) - \ophat{P}(s) \right]
				+ \left[ \op{P}(s) - \ophat{P}(s) \right] \, \sigma\, \sigma' \, \ophat{P}(s) 
			\right\|_{\bo(\cX)}
	\nn\\
	& \le \left( 
				\|\op{P}\|_{C[0,\tau]} 
				+ \|\ophat{P}\|_{C[0,\tau]} 
		\right) 
		\, b \, \| \op{P} - \ophat{P}\|_{C[0,\tau]} 
	\le 2\, r\, b \, \| \op{P} - \ophat{P}\|_{C[0,\tau]}\,.
	\label{eq:contraction-1-b}
\end{align}
Combining \er{eq:contraction-1-a} and \er{eq:contraction-1-b} and taking the supremum over $t\in[0,\tau]$,
\begin{align}
	\|\gamma_n (\op{P}) - \gamma_n (\ophat{P}) \|_{C[0,\tau]}
	& \le 2\, r\, \tau\, M_T^2\, b \, \| \op{P} - \ophat{P}\|_{C[0,\tau]}
	< \ts{\frac{1}{2}} \, \| \op{P} - \ophat{P}\|_{C[0,\tau]}\,.
	\nn
\end{align}
Hence, $\gamma_n:B_{C[0,\tau]}(r) \mapsinto B_{C[0,\tau]}(r)$ defines a contraction on $B_{C[0,\tau]}(r)$. Consequently, the Banach Fixed Point Theorem (for example, \cite[Theorem 5.1-4, p.303]{K:78}) implies that there exists a unique solution $\op{P}_n\in B_{C[0,\tau]}(r)$ of \er{eq:op-P-n-mild} for all $t\in[0,\tau]$ and $x\in\cX$. In order to conclude global uniqueness in $C([0,\tau];\Sigma(\cX))$, suppose there exists a second solution $\ophat{P}_n\in C([0,\tau]; \Sigma(\cX))$ of \er{eq:op-P-n} satisfying $\ophat{P}_n(0) = \op{P}_0$. That is, $\op{P}_n = \gamma_n(\op{P}_n)$ and $\ophat{P}_n = \gamma_n(\ophat{P}_n)$, where $\gamma_n$ is as per \er{eq:op-gamma-n}. Given any $x\in\cX$, using an inequality analogous to \er{eq:contraction-1-a},
\begin{align}
	\| \op{P}_n(t)\, x - \ophat{P}_n(t)\, x\|
	& \le M_T^2 \, b\,  \int_0^t \left( r + \|\ophat{P}_n\|_{C[0,\tau]} \right) \| \op{P}_n(s)\, x - \ophat{P}_n(s)\, x \| \, ds\,,
	\nn
\end{align}
where $\|\op{P}_n\|_{C[0,\tau]} \le r$ has been used. As $\ophat{P}_n\in\bo(\cX; C([0,\tau];\cX))$ by Lemma \ref{lem:spaces}, there exists a $K\in\R_{\ge 0}$ such that $\|\ophat{P}_n(\cdot)\, x\|_{C([0,\tau];\cX)} \le K\, \|x\|$ for all $x\in\cX$. Consequently, by Lemma \ref{lem:norms}, 
$
	\| \ophat{P}_n \|_{C[0,\tau]} = \|\ophat{P}_n\|_{C_0[0,\tau]} 
	= \sup_{\|x\|=1} \|\ophat{P}_n(\cdot)\, x\|_{C([0,\tau];\cX)} 
	\le K<\infty
$.
Combining these facts yields the inequality
\begin{align}
	\| \op{P}_n(t)\, x - \ophat{P}_n(t)\, x \|
	& 
	\le M_T^2 \, b\,  (r + K) \int_0^t  \| \op{P}_n(s)\, x - \ophat{P}_n(s)\, x \| \, ds
	\,.
	\nn
\end{align}
As $\op{P}_n - \ophat{P}_n\in C_0([0,\tau];\Sigma(\cX))$ is strongly continuous, the attendant function $\|\op{P}_n(\cdot)\, x - \ophat{P}_n(\cdot)\, x\|:[0,\tau]\mapsinto\R_{\ge 0}$ is continuous by definition. This admits a straightforward application of Gronwall's inequality, yielding 
$
	\|\op{P}_n(t)\, x - \ophat{P}_n(t)\, x\| \le 0 
$ for all $t\in[0,\tau]$, $x\in\cX$.
That is, $\op{P}_n = \ophat{P}_n$, so the asserted uniqueness is indeed global on $C([0,\tau];\Sigma(\cX))$. An analogous argument, using the same $M_T\in\R_{\ge 0}$ of \er{eq:M-T}, $\tau\in[0,T]$, and $r\in\R_{>0}$, implies the existence of a unique solution $\op{P}\in C_0([0,\tau];\Sigma(\cX))$ of \er{eq:op-P-mild}. The fact that \er{eq:op-limit} holds follows as per \cite[Lemma 2.1, p. 389]{BDDM:07}.
%
%
\end{proof}
\begin{assumption}
\label{ass:coercive}
There exists an operator $\op{M}\in\Sigma(\cX)$ such that the unique mild solution $\op{P}\in C_0([0,\tau_0]; \Sigma(\cX))$ of \er{eq:op-P} satisfying $\op{P}(0) = \op{M}$ that exists for some $\tau_0\in\R_{>0}$ by Theorem \ref{thm:existence-uniqueness-P} is such that $\op{P}(t) - \op{M}$ is coercive for all $t\in(0,\tau_0]$. That is, 
\begin{align}
	& \hspace{-1mm} \op{P}\in C_0([0,\tau_0]; \Sigma(\cX)) \cap C_0((0,\tau_0]; \Sigma_{\op{M}}(\cX)).
	\label{eq:op-P-coercive}
\end{align}
\if{false}

An operator $\op{M}\in\Sigma(\cX)$ exists such that the following holds:
\begin{enumerate}
\item The operator $\Gamma(\op{M})$ defined by
\begin{align}
	& 
	\Gamma(\op{M})
	\doteq \op{A}' \, \op{M} + \op{M}\, \op{A} + \op{M}\, \sigma\, \sigma' \op{M}  + \op{C}\,,
	\label{eq:op-Gamma}
\end{align}
is coercive, where $\op{A}$, $\sigma$, and $\op{C}$ are as per \er{eq:op-P}; and
\item 
\end{enumerate}

\fi
\end{assumption}
\vspace{1mm}
\begin{theorem}
\label{thm:optilde-P-existence}
Given any $\op{M}\in\Sigma(\cX)$ and $\tau_0\in\R_{>0}$ satisfying Assumption \ref{ass:coercive}, and any $\optilde{M}\in\Sigma_{\op{M}}(\cX)$, there exists a $\tau_1\in(0,\tau_0]$ such that a unique mild solution 
\begin{align}
	\optilde{P}
	& \in C_0([0,\tau_1]; \Sigma(\cX)) \cap C_0((0,\tau_1]; \Sigma_{\op{M}}(\cX))
	\label{eq:optilde-P-existence}
\end{align}
of \er{eq:op-P} satisfying $\optilde{P}(0) = \optilde{M}$ exists.
\end{theorem}
\vspace{1mm}
\begin{proof}
See Appendix \ref{app:optilde-P-proof}.
\end{proof}

\if{false}

\begin{remark}
By inspection of \er{eq:op-Gamma}, $\Gamma(\op{M})$ is not in $\Sigma(\cX)$ for arbitrary $\op{M}\in\Sigma(\cX)$ as $\op{A}$ is unbounded. However, by Assumption \ref{ass:op-A}, it is possible to choose $\op{M} = (\op{A}^{-1})'\, \op{M}_0 \, \op{A}^{-1}$ for any $\op{M}_0\in\Sigma(\cX)$, thereby yielding $\Gamma(\op{M})\in\bo(\cX)$ (via an extension from $\dom(\op{A})$ to $\cX$).
\end{remark}

\fi


\subsection{Auxiliary equations}
In proposing a max-plus dual space fundamental solution to the differential operator Riccati equation \er{eq:op-P}, two (additional) auxiliary operator differential equations are of interest. These equations, defined with respect to the same Hilbert spaces $\cX$ and $\cW$, are given by
\begin{align}
	\opdot{Q}(t)
	& = \op{A}' \, \op{Q}(t) + \op{P}(t) \, \sigma\, \sigma' \, \op{Q}(t)\,,
	\label{eq:op-Q}
	\\
	\opdot{R}(t)
	& = \op{Q}'(t) \, \sigma\, \sigma' \, \op{Q}(t)\,,
	\label{eq:op-R} 
\end{align}
in which $\op{A}:\dom(\op{A})\subset\cX\mapsinto\cX$ and $\sigma\in\bo(\cW;\cX)$ are defined as per \er{eq:op-P}. Also as per \er{eq:op-P}, any operator-valued functions $\op{Q}\in C_0([0,\tau];\bo(\cX))$ and $\op{R}\in C_0([0,\tau];\Sigma(\cX))$ satisfying
\begin{align}
	\op{Q}(t)\, x
	& = e^{\op{A}'\, t}\, \op{Q}(0)\, x + \int_0^t e^{\op{A}' \, (t-s)} \left[ \op{P}(s) \, \sigma\, \sigma' \, \op{Q}(s) \right] x\, ds\,,
	\label{eq:op-Q-mild}
	\\
	\op{R}(t)\, x
	& = \op{R}(0)\, x + \int_0^t \op{Q}(s)' \, \sigma\, \sigma' \, \op{Q}(s) \, x\, ds\,,
	\label{eq:op-R-mild}
\end{align}
for all $x\in\cX$, $t\in[0,\tau]$, $\tau\in\R_{>0}$, are defined to be {\em mild solutions} of \er{eq:op-Q} and \er{eq:op-R} (respectively) on $[0,\tau]$. With regard to the range of $\op{Q}$, note that $\op{Q}(t)\in\bo(\cX)$ is not self-adjoint by inspection of \er{eq:op-Q} or \er{eq:op-Q-mild}. That is, it may be shown that $\op{Q}\in C_0([0,\tau];\bo(\cX))$ (rather than $C_0([0,\tau];\Sigma(\cX))$). On the other hand, $\op{R}(t)\in\Sigma(\cX)$ is self-adjoint by inspection of \er{eq:op-R} or \er{eq:op-R-mild}. As per \er{eq:op-P} and \er{eq:op-P-n}, it is convenient to introduce analogous operator differential equations to \er{eq:op-Q} and \er{eq:op-R}, defined with respect to the Yosida approximation $\op{A}_n\in\bo(\cX)$ of $\op{A}$ for all $n\in\N$. In particular,
\begin{align}
	\opdot{Q}_n(t)
	& = \op{A}_n' \, \op{Q}_n(t) + \op{P}_n(t) \, \sigma\, \sigma'\, \op{Q}_n(t)\,,
	\label{eq:op-Q-n}
	\\
	\opdot{R}_n(t)
	& = \op{Q}_n'(t) \, \sigma\, \sigma' \, \op{Q}_n(t)\,.
	\label{eq:op-R-n}
\end{align}
Solutions of \er{eq:op-Q-n}, \er{eq:op-R-n} are any operator valued functions $\op{Q}_n\in C([0,t];\bo(\cX))$, $\op{R}_n\in C([0,t];\Sigma(\cX))$ that satisfies the corresponding integral equation, i.e.
\begin{align}
	\op{Q}_n(t)\, x
	& = e^{\op{A}_n'\, t}\, \op{Q}_n(0)\, x + \int_0^t e^{\op{A}_n' \, (t-s)} \left[ \op{P}_n(s) \, \sigma\, \sigma' \, \op{Q}_n(s) \right] x\, ds\,,
	\label{eq:op-Q-n-mild}
	\\
	\op{R}_n(t)\, x
	& = \op{R}_n(0)\, x + \int_0^t \op{Q}_n(s)' \, \sigma\, \sigma' \, \op{Q}_n(s) \, x\, ds\,,
	\label{eq:op-Q-R-n-mild}	
\end{align}
for all $x\in\cX$, where $\op{P}_n\in C_0([0,t];\Sigma(\cX))$ is the solution of \er{eq:op-P-n} (see also Appendix \ref{app:integral-forms}). Following the arguments used in the proofs of Theorems \ref{thm:existence-uniqueness-P}, existence of unique solutions of \er{eq:op-Q}, \er{eq:op-R}, \er{eq:op-Q-n}, \er{eq:op-R-n} can also be established for specific initial conditions. 

\begin{theorem}
\label{thm:existence-uniqueness-Q-R}
Given any $\op{M}\in\Sigma(\cX)$, and $\tau\in\R_{>0}$, $\op{P}\in C_0([0,\tau];\Sigma(\cX))$, $\op{P}_n\in C([0,\tau];\Sigma(\cX))$ as specified by Theorem \ref{thm:existence-uniqueness-P} with $\op{P}_0 = \op{M}$, there exists a $\tau_2\in(0,\tau]$ such that the operator differential equations \er{eq:op-Q}, \er{eq:op-R}, \er{eq:op-Q-n}, \er{eq:op-R-n} exhibit unique (mild) solutions 
\begin{align}
	&  \op{Q}\in C_0([0,\tau_2];\bo(\cX))\,, \ 	
	\op{Q}_n\in C([0,\tau_2];\bo(\cX))\,,
	\nn\\
	& \op{R}\in C_0([0,\tau_2];\Sigma(\cX))\,, \ 
	\op{R}_n\in C([0,\tau_2];\Sigma(\cX))\,,
	\nn
\end{align}
satisfying $\op{Q}(0) = -\op{M} = \op{Q}_n(0)$, $\op{R}(0) = \op{M} = \op{R}_n(0)$ for all $n\in\N$. Furthermore, for all $x\in\cX$,
\begin{align}
	&
	\lim_{n\rightarrow\infty} \op{Q}_n(\cdot)\, x = \op{Q}(\cdot)\, x\,,
	\
	\lim_{n\rightarrow\infty} \op{R}_n(\cdot)\, x = \op{R}(\cdot)\, x\,,
	\label{eq:op-Q-R-limits}
\end{align}
with the limits defined with respect to the Banach space $(C([0,\tau_2];\cX),\, \|\cdot\|_{C([0,\tau_2]; \cX)})$.
\end{theorem}
\begin{proof}
The proof is similar to that of Theorem \ref{thm:existence-uniqueness-P} and is omitted.
\end{proof}
 
\subsection{Common horizon of existence of solutions}
For the remainder, it is convenient to define a common horizon $\tau^*\in\R_{>0}$ of existence for the unique mild solutions $\op{P}$, $\optilde{P}$ of the operator differential Riccati equation \er{eq:op-P}, $\op{Q}$, $\op{R}$ of the auxiliary operator differential equations \er{eq:op-Q}, \er{eq:op-R}, and $\op{P}_n$, $\op{Q}_n$, $\op{R}_n$ of the corresponding Yosida approximation operator differential equations \er{eq:op-P-n}, \er{eq:op-Q-n}, \er{eq:op-R-n}. In particular, with $\tau_1(\op{M},\optilde{M}) \doteq \tau_1\in\R_{>0}$ as fixed by applying Theorem \ref{thm:optilde-P-existence} for any $\optilde{M}\in\Sigma_\op{M}(\cX)$ with $\op{M}\in\Sigma(\cX)$ fixed as per Assumption \ref{ass:coercive}, and $\tau_2(\op{M}) \doteq \tau_2\in\R_{>0}$ as fixed by applying Theorem \ref{thm:existence-uniqueness-Q-R} for the same $\op{M}\in\Sigma(\cX)$, define
\begin{align}
	\tau^* = \tau^*(\op{M},\optilde{M})
	& \doteq \tau_1(\op{M},\optilde{M})\wedge \tau_2(\op{M})\in\R_{>0}\,,
	\label{eq:tau-star}
\end{align}
where $\wedge$ denotes the $\min$ operation. That is, Theorems \ref{thm:existence-uniqueness-P} and \ref{thm:existence-uniqueness-Q-R} guarantee existence of unique $\op{P}$, $\optilde{P}$, $\op{Q}$, $\op{R}$,  and $\op{P}_n$, $\op{Q}_n$, $\op{R}_n$ on $[0,\tau^*]\subset\R_{\ge 0}$.


\section{Max-plus dual space fundamental solution semigroup}
\label{sec:fund}
A max-plus dual space fundamental solution semigroup for the operator differential Riccati equation \er{eq:op-P} is constructed by exploiting the semigroup property that attends the dynamic programming evolution operator of a related optimal control problem. In particular, by employing the Legendre-Fenchel transform \cite{R:74} of the value functional of an optimal control problem that encapsulates the particular mild solution $\op{P}$ of the operator differential Riccati equation \er{eq:op-P} initialized with $\op{P}(0) = \op{M}$ of Assumption \ref{ass:coercive}, a max-plus integral operator is defined in a corresponding max-plus dual space. It is demonstrated that this max-plus integral operator defines the aforementioned fundamental solution semigroup for the operator differential Riccati equation \er{eq:op-P}, which allows the realization of any solution $\optilde{P}$ of \er{eq:op-P} initialized with $\optilde{P}(0) = \optilde{M}\in\Sigma_\op{M}(\cX)$ as per Theorem \ref{thm:optilde-P-existence}. This construction generalizes the finite dimensional case documented in \cite{M:08}, and the infinite dimensional cases of \cite{DM:11,DM:12} in that it does not assume an explicit representation for the operator-valued solution of the operator differential Riccati equation \er{eq:op-P}.


\subsection{Optimal control problem}
An optimal control problem is defined with respect to the mild solution of the abstract Cauchy problem \cite{BDDM:07,CZ:95,P:83}
\begin{align}
	\dot\xi(t)
	& = \op{A}\, \xi(t) + \sigma\, w(t)\,,
	\label{eq:dynamics}
\end{align}
where $\xi(t)\in\cX$ denotes the state at time $t\in[0,\tau^*]$, $\tau^*\in\R_{>0}$ as per \er{eq:tau-star}, evolved from an initial state $\xi(0) = x\in\cX$ in the presence of an input signal $w\in\Ltwo([0,t];\cW)$. A mild solution of the abstract Cauchy problem \er{eq:dynamics} on the time interval $[0,\tau^*]$ is any function $\xi\in C([0,\tau^*];\cX)$ that satisfies
\begin{align}
	\xi(t)
	& = e^{\op{A}\, t}\, \xi(0) + \int_0^t e^{\op{A}\, (t-s)} \, \sigma\, w(s) \, ds\,,
	\label{eq:mild-dynamics}
\end{align}
(see for example \cite[Definition 3.1, p.129]{BDDM:07}, and also Appendix \ref{app:integral-forms}), where $e^{\op{A}\, t}\in\bo(\cX)$ denotes the corresponding element of the $C_0$-semigroup of bounded linear operators generated by $\op{A}$. 

\begin{remark}
Continuity of $\xi(\cdot)$ is in fact implied by \er{eq:mild-dynamics}, see for example \cite[Lemma 3.1.5, p.104]{CZ:95}. Indeed, given any $\xi(0) = x\in\cX$ and $w\in\Ltwo([0,t];\cW)$, the abstract Cauchy problem \er{eq:dynamics} has a unique {\em strong} solution which is also the mild solution (see for example \cite[Definition 3.1, Proposition 3.1, p.129--130]{BDDM:07}). 
\end{remark}

With the dynamics specified and interpreted as per \er{eq:dynamics} and \er{eq:mild-dynamics} respectively, the value functional $W^z:[0,\tau^*]\times\cX\mapsinto\R$ of the optimal control problem of interest is defined for each $z\in\cX$ by
\begin{align}
	W^z(t,x)
	& \doteq \sup_{w\in\Ltwo([0,t];\cW)} J^z(t,x;w)\,,
	\label{eq:value}
\end{align}
where the payoff $J^z:[0,\tau^*]\times\cX\times\Ltwo([0,\tau^*];\cW)\mapsinto\R$, $z\in\cX$ fixed, is defined with respect to the unique mild solution \er{eq:mild-dynamics} corresponding to $\xi(0) = x\in\cX$ by
\begin{align}
	J^z(t,x;w)
	& \doteq \int_0^t \demi \langle \xi(s),\, \op{C}\, \xi(s) \rangle - \demi \|w(s)\|^2 \, ds
	+ \psi(\xi(t),z).
	\label{eq:payoff}
\end{align}
Here, the terminal payoff $\psi(\cdot,z):\cX\mapsinto\R$ is defined with respect to the same operator $\op{M}\in\Sigma(\cX)$ of Assumption \ref{ass:coercive} by
\begin{align}
	\psi(x,z)
	& \doteq \demi\langle x-z,\, \op{M}\, (x-z) \rangle\,.
	\label{eq:terminal}
\end{align}
Solutions of the operator differential Riccati equation \er{eq:op-P}, and the auxiliary operator differential equations \er{eq:op-Q}, \er{eq:op-R}, are fundamentally related to the optimal control problem of \er{eq:value}. To explore and exploit this relationship, for each $x,z\in\cX$, $s\in[0,t]$, define
\begin{align}
	\op{F}_t(s) \, x
	& \doteq \sigma'\, (\op{P}(t - s)\, x + \op{Q}(t - s)\, z)\,,
	\label{eq:control}
\end{align}
where $\op{P}\in C_0([0,\tau^*];\Sigma(\cX))$ and $\op{Q}\in C_0([0,\tau^*];\bo(\cX))$ denote the unique mild solutions of \er{eq:op-P} and \er{eq:op-Q} satisfying $\op{P}(0) = \op{M}$ and $\op{Q}(0) = -\op{M}$ respectively, as per Theorems \ref{thm:existence-uniqueness-P} and \ref{thm:existence-uniqueness-Q-R}. The map \er{eq:control} can be regarded as a feedback for the abstract Cauchy problem \er{eq:dynamics}, yielding the closed-loop abstract Cauchy problem
\begin{align}
	\dot \xi(s) 
	& = \left( \op{A} + \sigma\, \op{F}_t(s) \right) \xi(s)\,, \quad s\in[0,t]\,,
	\label{eq:closed-loop-Cauchy}
\end{align}
where $\xi(0) = x\in\cX$ and $t\in[0,\tau^*]$.
\begin{theorem}
\label{thm:optimal}
Given any $t\in[0,\tau^*]$, the closed-loop abstract Cauchy problem \er{eq:closed-loop-Cauchy} has a unique mild solution $\xi^*\in C([0,t];\cX)$. Furthermore, the input $w^*\in C([0,t];\cW)$ defined by
\begin{align}
	w^*(s) 
	& \doteq \op{F}_t(s)\, \xi^*(s)
	= \sigma' (\op{P}(t - s)\, \xi^*(s) + \op{Q}(t - s)\, z)
	\label{eq:w-star}
\end{align}
is optimal with respect to \er{eq:value}, \er{eq:payoff}, with
\begin{equation}
	\begin{aligned}
	J^z(t,x;w)
	& \le J^z(t,x;w^*)
	= W^z(t,x) 
	\\
	& = \demi \langle x,\, \op{P}(t)\, x + \langle x,\, \op{Q}(t) \, z \rangle + \demi \langle z,\, \op{R}(t)\, z\rangle
	\end{aligned}
	\label{eq:optimal}
\end{equation}
for all $w\in\Ltwo([0,t];\cW)$, $x\in\cX$.
\end{theorem}
\begin{proof}
Fix any $t\in[0,\tau^*]$, where $\tau^*\in\R_{>0}$ is as per \er{eq:tau-star}. The abstract Cauchy problem \er{eq:closed-loop-Cauchy} exhibits a unique mild solution $\xi^*\in C([0,t];\cX)$ via a straightforward modification of \cite[Proposition 6.1, p.409]{BDDM:07}. The fact that input $w^*$ defined by \er{eq:w-star} is optimal follows by a modification of the proof of \cite[Proposition 6.2, p.409]{BDDM:07}. In particular, let $\op{P}_n, \op{R}_n\in C([0,t];\Sigma(\cX))$, $\op{Q}_n\in C([0,t];\bo(\cX))$, denote the unique solutions of \er{eq:op-P-n}, \er{eq:op-Q}, \er{eq:op-R}, corresponding to the Yosida approximation $\op{A}_n$ of $\op{A}$ that exists on interval $[0,\tau^*]$ by Theorems \ref{thm:existence-uniqueness-P} and \ref{thm:existence-uniqueness-Q-R}. Similarly, let $\xi_n\in C([0,t];\cX)$ denote the unique solution of the abstract Cauchy problem \er{eq:dynamics} with $\op{A}$ replaced with $\op{A}_n$ for arbitrary $w\in\Ltwo([0,t];\cW)$. 
Define $\pi_n:[0,t]\mapsinto\R$ by
\begin{align}
	\pi_n(s)
	& \doteq p_n(s) + q_n(s) + r_n(s)
	\label{eq:pi}
\end{align}
where $p_n,q_n,r_n:[0,t]\mapsinto\R$ are given by
\begin{align}
	p_n(s)
	& \doteq \demi \langle \xi_n(s),\, \op{P}_n(t-s)\, \xi_n(s) \rangle\,,
	\label{eq:p}
	\\
	q_n(s)
	& \doteq \langle \xi_n(s),\, \op{Q}_n(t-s)\, z \rangle\,,
	\label{eq:q}
	\\
	r_n(s)
	& \doteq \demi \langle z,\, \op{R}_n(t-s)\, z \rangle\,.
	\label{eq:r}
\end{align}
Differentiating (formally) and applying \er{eq:op-P}, \er{eq:op-Q}, and \er{eq:op-R}, it is straightforward to show that
\begin{align}
	\dot{p}_n(s)
	& = 
	-\demi \langle \xi_n(s),\, \left[ \op{C} + \op{P}_n(t-s) \, \sigma\, \sigma' \, \op{P}_n(t-s) \right]\, \xi_n(s) \rangle
	\nn\\
	& \qquad + \langle w(s),\, \sigma' \op{P}_n(t-s)\, \xi_n(s) \rangle\,,
	\label{eq:p-dot}
	\\
	\dot{q}_n(s)
	& = - \langle \xi_n(s),\, \op{P}_n(t-s)\, \sigma\, \sigma' \, \op{Q}_n(t-s)\, z \rangle
	\nn\\
	& \qquad + \langle w(s),\, \sigma' \op{Q}_n(t-s)\, z \rangle\,,
	\label{eq:q-dot}
	\\
	\dot{r}_n(s)
	& = -\demi \langle z,\, \op{Q}_n(t-s)' \, \sigma\, \sigma' \, \op{Q}_n(t-s) \, z \rangle
	\label{eq:r-dot}
\end{align}
Define $\ol{w}(s) \doteq \sigma' \left( \op{P}_n(t-s) \, \xi_n(s) + \op{Q}_n(t-s)\, z \right)$ for all $s\in[0,t]$. Differentiation of \er{eq:pi}, substitution of \er{eq:p-dot}, \er{eq:q-dot}, \er{eq:r-dot}, followed by completion of squares, yields
\begin{align}
	\dot\pi_n(s)
	& = - \left[ \demi \langle \xi_n(s)\, \op{C}\, \xi_n(s) \rangle - \demi \|w(s)\|^2 \right]  - \demi \|w(s) - \ol{w}(s)\|^2
	\label{eq:pi-dot}
\end{align}
for all $s\in[0,t]$. As $\op{P}_n\in C([0,t];\Sigma(\cX))$ and $\op{Q}_n\in C([0,t];\bo(\cX))$, note that $\ol{w}\in C([0,t];\cW)\subset\Ltwo([0,t];\cW)$ by definition. Recalling that $\op{P}_n(0) = \op{M} = \op{R}_n(0)$ and $\op{Q}_n(0) = -\op{M}$, \er{eq:pi} implies that
\begin{align}
	\pi_n(t)
	& = \demi \langle \xi_n(t), \, \op{M} \, \xi_n(t) \rangle - \langle \xi_n(t),\, \op{M} \, z \rangle
			+ \demi \langle z,\, \op{M}\, z \rangle
	= \psi(\xi_n(t),z)\,,
	\label{eq:pi-t}
\end{align}
where $\psi$ is the terminal payoff \er{eq:terminal}. Similarly, as $\xi_n(0) = x$,
$
	\pi_n(0)
	= \demi \langle x,\, \op{P}_n(t) \, x \rangle + \langle x,\, \op{Q}_n(t)\, z \rangle + \demi \langle z,\, \op{R}_n(t)\, z\rangle
$.
Note that the limit as $n\rightarrow\infty$ of $\pi_n(0)$ is well-defined by the corresponding limits of $\op{P}_n$, $\op{Q}_n$, $\op{R}_n$ defined by the Yosida approximation, with
\begin{align}
	\pi_\infty(0)
	& \doteq \lim_{n\rightarrow\infty} \pi_n(0)
	= \demi\, \langle x,\, \op{P}(t)\, x \rangle + \langle x,\, \op{Q}(t)\, z \rangle + \demi\, \langle z,\, \op{R}(t)\, z \rangle\,.
	\label{eq:pi-infty}
\end{align}
Meanwhile, integrating \er{eq:pi-dot} with respect to $s\in[0,t]$ and applying \er{eq:pi-t} yields
$
	\pi_n(0)
	= \int_0^t \demi\langle \xi_n(s),\, \op{C}\, \xi_n(s) \rangle - \demi \|w(s)\|^2 \, ds + \psi(\xi_n(t),\, z)
	+ \demi \int_0^t \|w(s) - \ol{w}(s)\|^2\, ds
$.
Taking the limit as $n\rightarrow\infty$ and applying \er{eq:payoff} and the definition \er{eq:pi-infty} of $\pi_\infty(0)$ yields
$
	\pi_\infty(0) - \demi \int_0^t \|w(s) - \ol{w}(s)\|^2\, ds
	= J^z(t,x;w)
$.
Finally, taking the supremum over $w\in\Ltwo([0,t];\cW)$ and applying the right-hand equality in \er{eq:pi-infty} yields
\begin{align}
	W^z(t,x) & = \sup_{w\in\Ltwo([0,t];\cW)} J^z (t,x;w)
	\nn\\
	& = \pi_\infty(0)
	= \demi \langle x,\, \op{P}(t)\, x \rangle + \langle x,\, \op{Q}(t)\, z\rangle + \demi \langle z,\, \op{R}(t)\, z\rangle\,,
	\nn
\end{align}
in which the optimal input is $w^*= \ol{w}$, as per \er{eq:w-star}.
%
%
\end{proof}

Theorem \ref{thm:optimal} is crucial to the development of a max-plus fundamental solution to the operator differential Riccati equation \er{eq:op-P}. In particular, it demonstrates that the unique mild solution $\op{P}\in C_0([0,\tau^*];\Sigma(\cX))$ of \er{eq:op-P}, initialized with $\op{P}(0) = \op{M}$ of Assumption \ref{ass:coercive}, may be propagated forward in time via propagation of the value function $W^z(t,\cdot)$ of \er{eq:value} with respect to its time horizon $t\in[0,\tau^*]$. This is significant as propagation of $W^z(t,\cdot)$ is possible via the dynamic programming \cite{BCD:97,B:57} evolution operator. In particular, $W^z$ may be written as
\begin{align}
	W^z(t,x)
	& = (\op{S}_t \, \psi(\cdot,z))(x)
	\label{eq:value-op-S}
\end{align}
for all $t\in[0,\tau^*]$, $x\in\cX$, where $\op{S}_t$ denotes the aforementioned dynamic programming evolution operator. This operator is defined by
\begin{align}
	(\op{S}_t\, \Psi)(x)
	& \doteq \hspace{0mm} \sup_{w\in\Ltwo([0,t];\cW)} \left\{ 
							\int_0^t \demi\, \langle \xi(s),\, \op{C}\, \xi(s) \rangle - \demi\, \|w(s)\|^2 \, ds
							+ \Psi(\xi(t))
			\right\},
	\label{eq:op-S}
\end{align}
where $\xi(\cdot)$ is the unique mild solution of \er{eq:dynamics} satisfying $\xi(0) = x$. It satisfies the semigroup property
\begin{align}
	\op{S}_{t+s}
	& = \op{S}_s \, \op{S}_t = \op{S}_t\, \op{S}_s
	\label{eq:op-S-semigroup}
\end{align}
for all $s,t\in[0,\tau^*]$, $s+t\in[0,\tau^*]$, which in combination with \er{eq:value-op-S}, allows $W^z(t,\cdot)$ to be propagated to longer time horizons. 

\begin{remark}
\label{rem:semigroup}
Although the dynamic programming evolution operator of \er{eq:op-S} satisfies the semigroup property \er{eq:op-S-semigroup}, the horizon-indexed set of these operators $\left\{ \op{S}_t \, \bigl| \, t\in[0,\tau^*] \right\}$, equipped with the binary operation of operator composition, does not formally define a semigroup for $\tau^*<\infty$. In particular, note that
\begin{align}
	t>\tau^*/2
	& \quad\Longrightarrow\quad
	\op{S}_{t} \,\op{S}_{t} = \op{S}_{2\, t}\not\in \left\{ \op{S}_t \, \bigl| \, t\in[0,\tau^*] \right\}.
	\nn
\end{align}
However, it is possible to define a set of horizon-operator pairs, with an associated binary operation, that always defines a semigroup. In particular, consider a family of generic horizon-indexed operators $\{\op{F}_t\}_{t\in[0,\tau]}$ defined for some $\tau\in\R_{>0}$ and satisfying $\op{F}_0 = \op{I}$ (the identity operator) and the semigroup property
\begin{align}
	& \op{F}_{t+s} = \op{F}_s\, \op{F}_t = \op{F}_t\, \op{F}_s
	\label{eq:op-F-semigroup}
\end{align}
for all $s,t\in[0,\tau]$ for $\tau\in\R_{>0}$. Define the pair
\begin{align}
	\sigma_{\tau} \left(\op{F}\right) \doteq (\Gamma_\tau(\op{F}), \, \circ_\tau)
	\label{eq:sigma-semigroup}
\end{align}
where $\Gamma_\tau(\op{F})$ is a set of horizon-operator pairs, and $\circ_\tau:\Gamma_\tau(\op{F})\times\Gamma_\tau(\op{F})\mapsinto\Gamma_\tau(\op{F})$ is a binary operation, defined in turn by
\begin{align}
	\Gamma_{\tau}(\op{F})
	& \doteq \left\{ (t,\, \op{F}_t)\, \biggl| \, t\in[0,\tau]
	\right\},
	\nn\\
	(s,\, \op{F}_s) \circ_\tau (t,\, \op{F}_t)
	& \doteq \left\{ \ba{cl}
		(s + t,\, \op{F}_s\, \op{F}_t)\,,
		& s + t\in[0,\tau)\,,
		\\
		(\tau, \op{F}_{\tau})\,, 
		& s + t\in[\tau,\infty)\,.
	\ea \right.
	\nn
\end{align}
It is straightforward to show that the pair $\sigma_{\tau}(\op{F})$ of \er{eq:sigma-semigroup} defines a semigroup, as
\begin{align}
	& \theta\circ_\tau(\phi\circ_\tau\chi) = (\theta\circ_\tau\phi)\circ_\tau\chi\,,
	&& \forall \ \theta,\phi,\chi\in\Gamma_\tau(\op{F})\,.
	\nn
\end{align}
Furthermore, $\sigma_{\tau^*}(\op{F})$ also comes equipped with the identity element $\mathbf{1} \doteq (0,\, \op{F}_0)\in\Gamma_\tau$, so that $\mathbf{1} \circ \theta = \theta = \theta \circ \mathbf{1}$ for all $\theta\in\Gamma_\tau$.
In the special case of the dynamic programming evolution operator $\op{S}_t$ of \er{eq:op-S}, the dynamic programming principle described by \er{eq:op-S-semigroup} immediately implies that the pair $\sigma_{\tau^*}(\op{S})$ defines a semigroup via \er{eq:sigma-semigroup}.
\end{remark}


\subsection{Max-plus dual space representation of $W^z$}
The semigroup property \er{eq:op-S-semigroup} describes how the value functional $W^z(t,\cdot)$ of \er{eq:value} can be propagated from any initial horizon $t\in[0,\tau^*]$ to any final longer time horizon $t+s\in[0,\tau^*]$, $s\in[0,\tau^*-t]$, via dynamic programming. As this value functional is identified with the operator differential Riccati equation solution $\op{P}$ of \er{eq:op-P} via Theorem \ref{thm:optimal}, this value functional propagation corresponds to evolution of $\op{P}(t)$ from its initial condition $\op{P}(0) = \op{M}\in\Sigma(\cX)$ satisfying Assumption \ref{ass:coercive}. By appealing to semiconvex duality \cite{R:74} of the value functional, and max-plus linearity of the dynamic programming evolution operator, this evolution can be represented via a dual space evolution operator that is defined independently of the terminal payoff $\psi$ of \er{eq:terminal}, and hence the initial data $\op{M}\in\Sigma(\cX)$. The dual space evolution operator obtained is subsequently shown to propagate the solution $\optilde{P}(t)$ of the operator differential Riccati equation \er{eq:op-P} from any arbitrary initial condition $\optilde{P}(0) = \optilde{M}\in\Sigma_{\op{M}}(\cX)$ satisfying the conditions of Theorem \ref{thm:optilde-P-existence}. 

This development relies on concepts and results from convex and idempotent analysis. In particular, semiconvex duality \cite{R:74} is introduced using operators defined with respect to the max-plus algebra, c.f. \cite{M:06,M:08}, etc. The max-plus algebra is a commutative semifield over $\R^{-} \doteq \R\cup\{-\infty\}$ equipped with the addition and multiplication operations $\oplus$ and $\otimes$ that are defined by $a\oplus b \doteq \max(a,b)$ and $a\otimes b \doteq a + b$. It is also an idempotent semifield as $\oplus$ is an idempotent operation (i.e. $a\oplus a = a$) with no inverse. The respective spaces $\semiconvex{\op{K}}(\cX)$ and $\semiconcave{\op{K}}(\cX)$ of semiconvex and semiconcave functionals are defined with respect to $\op{K}\in\Sigma(\cX)$ by
\begin{align}
	\semiconvex{\op{K}}(\cX)
	& \doteq
	\left\{ f:\cX\mapsinto\R^- \ \left| \ \ba{c} 
		\text{$f$ closed,}
		\\
		f + \demi \, \langle \cdot, \, \op{K}\, \cdot \rangle 
	 	\text{ convex}
		\ea \right. 
	\right\}\,,
	\label{eq:semiconvex}
	\\
	\semiconcave{\op{K}}(\cX)
	& \doteq
	\left\{ f:\cX\mapsinto\R^- \ \left| \ \ba{c} 
		\text{$f$ closed,}
		\\
		f - \demi \, \langle \cdot, \, \op{K}\, \cdot \rangle  
	 	\text{ concave}
		\ea \right. 
	\right\}\,.
	\label{eq:semiconcave}
\end{align}
It may be shown that $\semiconvex{\op{K}}(\cX)$ is a max-plus vector space of functionals defined on $\cX$, see \cite{M:06} for the analogous details in the finite dimensional case. Semiconvex duality \cite{R:74} is formalized as follows, in which max-plus integration of a functional $f$ over $\cX$ is defined by $\int_{\cX}^{\oplus} f(z)\, dz\doteq \sup_{z\in\cX} f(z)$. 
\begin{theorem}
\label{thm:dual} 
Fix $\op{K}\in\Sigma(\cX)$ satisfying $\op{K}<-\op{M}$, where $\op{M}\in\Sigma(\cX)$ is as per Assumption \ref{ass:coercive}. Then, for any $\phi\in\skx$,
\begin{align}
	\phi
	& = \op{D}_{\psi}^{-1} \, a \in \skx\,,
	\quad
	a = \op{D}_{\psi}\, \phi \in \semiconcave{\op{K}}(\cX)\,,
	\label{eq:primal-dual}
\end{align}
where $\psi$ is the quadratic bi-functional \er{eq:terminal}, and $\op{D}_{\psi}$, $\op{D}_{\psi}^{-1}$ denote respectively the semiconvex dual and inverse dual operators \cite{R:74} defined by
\begin{align}
	\op{D}_{\psi} \, \phi  
	= ( \op{D}_{\psi} \, \phi )(\cdot)
	& \doteq - \int_{\cX}^{\oplus} \psi(x,\cdot)\otimes (-\phi(x)) \, dx\,,
	\label{eq:op-D}
	\\\displaybreak[1]
	\op{D}_{\psi}^{-1} \, a 
	= ( \op{D}_{\psi}^{-1} \, a)(\cdot)
	& \doteq \int_{\cX}^{\oplus} \psi(\cdot,z) \otimes a(z) \, dz\,.
	\label{eq:op-D-inverse}
\end{align}
\end{theorem}

\vspace{1mm}
\begin{proof}
\er{eq:primal-dual} follows by Lemma \ref{lem:quadratic} and \cite[Theorem 5]{R:74}.
\end{proof}

In order to demonstrate that the semiconvex dual of $\op{S}_{t}\, \psi(\cdot,z)$ is well-defined for each $t\in(0,\tau^*]$ and $z\in\cX$, define $\op{K}_t\in\Sigma(\cX)$ by $\op{K}_t \doteq - \alpha \, \op{P}(t) - (1 - \alpha) \, \op{M}$, with $\alpha\in(0,1)$ fixed, where $\op{P}(t)$ and $\op{M}$ are as per Assumption \ref{ass:coercive}. Identity \er{eq:value-op-S}, definition \er{eq:op-S}, and Theorem \ref{thm:optimal} imply that
\begin{align}
	\left( \op{S}_{t} \, \psi(\cdot,z) \right) (x) & + \demi \, \langle x,\, \op{K}_t \, x \rangle
	= W^{z}(t,x) + \demi \, \langle x,\, \op{K}_t \, x \rangle
	\nn\\
	& = \demi \, \left\langle x,\, \left( \op{P}(t) + \op{K}_t \right) \, x \right\rangle + 
	\langle x,\, \op{Q}(t) \, z \rangle + 
	\demi \, \langle z,\, \op{R}(t) \, z \rangle\,.
	\label{eq:convexity-check}
\end{align}
%
With $t\in(0,\tau^*]$, note that $\op{P}(t) > \op{M}$, so that $\op{P}(t) + \op{K}_t = (1 - \alpha) \left( \op{P}(t) - \op{M} \right) > 0$, and $-\op{K}_t - \op{M} = \alpha \, \left( \op{P}(t) - \op{M} \right) > 0$. That is, $\op{K}_t$ is self-adjoint and satisfies $-\op{P}(t) < \op{K}_t < -\op{M}$. Hence, the right-hand side of \er{eq:convexity-check} is the sum of a non-negative quadratic functional and an affine functional. As any non-negative quadratic functional is convex by assertion {\em (ii)} of Lemma \ref{lem:quadratic}, and any affine functional is convex by definition, the right-hand side of \er{eq:convexity-check} is also convex. Hence, 
\begin{align}
	\op{S}_{t} \, \psi(\cdot,z) \in \scf{\op{K}_{t}}{\cX}\,.
	\label{eq:convex-evolution}
\end{align}
for all $t\in(0,\tau^*]$. Also note that as $\op{S}_{t} \, \psi(\cdot,z)$ is closed by Theorem \ref{thm:optimal} and assertion {\em (i)} of Lemma \ref{lem:quadratic}. Consequently, Theorem \ref{thm:dual} implies that the semiconvex dual of $\op{S}_{t}\, \psi(\cdot,z)$ is well-defined for any $z\in\cX$. Denote this dual by the functional $B_{t}(\cdot,z):\cX\mapsinto\R$ for each $z\in\cX$ fixed, so that \er{eq:primal-dual} yields
\begin{align}
	\left( \op{S}_{t} \, \psi(\cdot,z) \right)(x)
	& = ( \op{D}_{\psi}^{-1} \, B_{t}(\cdot,z) ) \, (x)\,,
	\label{eq:primal-W}
	\\
	B_{t}(y,z)
	& = 
	( \op{D}_{\psi} \, \op{S}_{t} \, \psi(\cdot,z) ) \, (y)\,.
	\label{eq:dual-W}
\end{align}
for all $t\in(0,\tau^*]$, $x,y,z\in\cX$. Theorem \ref{thm:optimal} also ensures that an explicit quadratic form for the functional $B_{t}(\cdot,z)$ is inherited from $\op{S}_{t} \, \psi(\cdot,z)$, as formalized below.
\begin{lemma}
\label{lem:B-explicit}
$B_{t}:\cX\times\cX\mapsinto\R^-$ is a quadratic functional given explicitly by
\begin{align}
	B_{t}(y,z)
	& = \demi \langle y,\, \op{B}_{t}^{1,1}\, y\rangle + \langle z,\, \op{B}_{t}^{1,2}\, y \rangle + 
	\demi \langle z,\, \op{B}_{t}^{2,2}\, z \rangle
	\label{eq:B-explicit} 
\end{align}
for all $t\in(0,\tau^*]$, $y, z\in\cX$, where $\op{B}_{t}^{1,1},\, \op{B}_{t}^{2,2}\in\Sigma(\cX)$, $\op{B}_{t}^{1,2}\in\bo(\cX)$  are defined by
\begin{align}
	\op{B}_{t}^{1,1}
	& \doteq - \op{M} - \op{M} \left(\op{P}(t) - \op{M} \right)^{-1} \op{M}\,,
	\label{eq:op-B-11}
	\\
	\op{B}_{t}^{1,2}
	& \doteq - \op{Q}(t)' \left( \op{P}(t) - \op{M} \right)^{-1}\, \op{M}\,,
	\label{eq:op-B-12}
	\\
	\op{B}_{t}^{2,2}
	& \doteq - \op{Q}(t)' \left( \op{P}(t) - \op{M} \right)^{-1} \op{Q}(t) + \op{R}(t)\,.
	\label{eq:op-B-22}
\end{align}
\end{lemma}
\begin{proof}
With $\tau^*\in\R_{>0}$ fixed as per \er{eq:tau-star}, recall by \er{eq:convex-evolution} , \er{eq:primal-W} and \er{eq:dual-W} that $B_{t}(\cdot,z)$ is the well-defined dual of $\op{S}_{t}\, \psi(\cdot,z)$. In particular, $B_{t}(y,z)$ is finite for all $t\in(0,\tau^*]$, $y,z\in\cX$. Applying \er{eq:op-D}, \er{eq:dual-W}, and Theorem \ref{thm:optimal}, $B_{t}(y,z) = - \int_{\cX}^{\oplus} \pi_{t}^{y,z}(x)\, dx$
for all $t\in(0,\tau^*]$, $y,z\in\cX$, where
\begin{align}
	\pi_{t}^{y,z}(x) 
	& \doteq \psi(x,y) \otimes \left( - (\op{S}_{t} \, \psi(\cdot,z))(x) \right)
	\nn\\
	& = \demi \langle x - y,\, \op{M}\, (x - y) \rangle - \demi \langle x,\, \op{P}(t) \, x \rangle
	- \langle x,\, \op{Q}(t) \, z \rangle - \demi\langle z,\, \op{R}(t) \, z \rangle
	\nn\\
	& = \demi \langle x,\, (\op{M} - \op{P}(t))\, x \rangle + \langle x, \, -( \op{M} \, y + \op{Q}(t)\, z) \rangle
	+ \demi \langle y,\, \op{M} \, y \rangle - \demi \langle z,\, \op{R}(t) \, z \rangle
	\nn\\
	& = b(x) + \demi \langle y,\, \op{M} \, y \rangle - \demi \langle z,\, \op{R}(t) \, z \rangle\,,
	\nn
\end{align}
and $b:\cX\mapsinto\R$ is the quadratic functional defined by $b(x) \doteq \demi \langle x,\, (\op{M} - \op{P}(t))\, x \rangle + \langle x, \, -( \op{M} \, y + \op{Q}(t)\, z) \rangle$ for all $x\in\cX$. That is,
\begin{align}
	B_{t}(y,z) 
	& = - \demi\langle y,\, \op{M}\, y \rangle + \demi \langle z,\, \op{R}(t) \, z \rangle -
	\int_{\cX}^{\oplus} b(x) \, dx\,.
	\label{eq:B}
\end{align}
Note that $\op{M}-\op{P}(t)\in\Sigma(\cX)$, while $-( \op{M} \, y + \op{Q}(t)\, z)\in\cX$. As $B_{t}$ is finite as previously indicated, Lemma \ref{lem:pseudo} implies that the supremum in \er{eq:B} is attained at $x^{*} = - (\op{P}(t) - \op{M})^{-1} (\op{M} \, y + \op{Q}(t)\, z)$, with
\begin{align}
	\int_{\cX}^{\oplus} b(x)\, dx
	= 
	b(x^{*})
	& = \demi\langle \op{M} \, y + \op{Q}(t)\, z,\, (\op{P}(t) - \op{M})^{-1} (\op{M} \, y + \op{Q}(t)\, z) \rangle
	\nn\\
	& = \demi\langle y,\, \op{M} \, (\op{P}(t) - \op{M})^{-1} \op{M} \, y \rangle +
	\langle z,\, \op{Q}(t)' \, (\op{P}(t) - \op{M})^{-1} \, \op{M}\, y \rangle
	\nn\\
	& \hspace{3cm} + 
	\demi \langle z,\, \op{Q}(t)' \, (\op{P}(t) - \op{M})^{-1}\, \op{Q}(t)\, z \rangle\,,
	\nn
\end{align}
where it may be noted that the inverse (rather than the pseudo-inverse) $(\op{P}(t)-\op{M})^{-1}\in\Sigma(\cX)$ exists as $\op{P}(t)-\op{M}\in\Sigma(\cX)$ is coercive for all $t\in(0,\tau^*]$ by Assumption \ref{ass:coercive}, see \cite[Examples A.4.2 and A.4.3, p.609]{CZ:95}. So, recalling \er{eq:B},
\begin{align}
	B_{t}(y,z)
	= & \demi \langle y,\, -(\op{M} + \op{M} \, (\op{P}(t) - \op{M})^{-1} \op{M}) \, y \rangle
	+ \langle z,\, - \op{Q}(t)' \, (\op{P}(t) - \op{M})^{-1} \, \op{M}\, y \rangle
	\nn\\
	& \qquad\qquad\qquad 
	+ \langle z,\, (- \op{Q}(t)' \, (\op{P}(t) - \op{M})^{-1}\, \op{Q}(t) + \op{R}(t))\, z\rangle\,,
	\nn
\end{align}
which is as per \er{eq:B-explicit} via definitions \er{eq:op-B-11}, \er{eq:op-B-12}, \er{eq:op-B-22}. Hence, as $\op{M}\in\Sigma(\cX)$, $\op{Q}\in C_0([0,\tau^*];\bo(\cX))$ and $\op{R}\in C_0([0,\tau^*];\Sigma(\cX))$,  definitions \er{eq:op-B-11}, \er{eq:op-B-12} and \er{eq:op-B-22} imply that $\betaOO{t}, \, \betaTT{t}\in\Sigma(\cX)$ and $\betaOT{t}\in\bo(\cX)$ for all $t\in(0,\tau^*]$.
\end{proof}
%
%


\subsection{Fundamental solution semigroup}
The functional $B_t$ of \er{eq:B-explicit} may be used as the kernel in defining a max-plus integral operator $\opBmp{t}$ on the dual-space of functionals generated by the semiconvex dual operator $\op{D}_{\psi}$ of \er{eq:primal-dual}. Specifically,
\begin{align}
	\opBmp{t}\, a
	=
	\left( \opBmp{t} \, a \right)(\cdot)
	& \doteq 
	\int_{\cX}^{\oplus} B_{t}(\cdot,z) \otimes a(z) \, dz
	\label{eq:op-B}
\end{align}
for all $t\in(0,\tau^*]$. This operator will be identified as an element of the new max-plus dual space fundamental solution semigroup for the operator differential Riccati equation \er{eq:op-P}. To this end, fix any operator $\optilde{M}\in\Sigma_{\op{M}}(\cX)$ as per Theorem  \ref{thm:optilde-P-existence}, and define the quadratic functional $\fundterm:\cX\mapsinto\R$ by
\begin{align}
	\fundterm(x)
	& \doteq \demi\, \langle x,\, \opMtilde\, x \rangle\,.
	\label{eq:fundterm}
\end{align}
By replacing the terminal payoff $\psi$ of \er{eq:terminal} with $\fundterm$ in the value functional $W^z$ of \er{eq:value}, note that the unique solution $\optilde{P}$ of the operator differential Riccati equation \er{eq:op-P} initialized with $\optilde{P}(0) = \optilde{M}$ and defined on $[0,\tau^*]$ may be characterized in an analogous way to Theorem \ref{thm:optimal}. That is, $\optilde{P}(t)$ may be identified with the propagated value functional $\op{S}_{t}\, \fundterm$ for all $t\in[0,\tau^*]$. Furthermore, this value functional can be represented in terms of the max-plus integral operator $\opBmp{t}$ of \er{eq:op-B} for all $t\in(0,\tau^*]$, as summarized by the following theorem.
\begin{theorem}
\label{thm:fund-general}
The value functional $\op{S}_{t} \, \fundterm$ defined via evolution operator $\op{S}_{t}$ of \er{eq:op-S} and terminal payoff functional $\fundterm$ of \er{eq:fundterm} may be represented equivalently by
\begin{align}
	( \op{S}_{t}\, \fundterm ) (x) 
	& = \demi \langle x,\, \optilde{P}(t) \, x \rangle
	= (\op{D}_{\psi}^{-1} \, \opBmp{t} \, \op{D}_{\psi}\, \fundterm)(x)\,, \quad x\in\cX,
	\label{eq:Wtilde-dual}
\end{align}
for all $t\in(0,\tau^*]$, where $\optilde{P}$ is the solution of the operator differential Riccati equation \er{eq:op-P} satisfying $\optilde{P}(0) = \optilde{M}$ as per Theorem \ref{thm:optilde-P-existence}, and $\op{D}_{\psi}$, $\op{D}_{\psi}^{-1}$, $\opBmp{t}$ denote the semiconvex dual operators \er{eq:op-D}, \er{eq:op-D-inverse}, and the max-plus integral operator \er{eq:op-B}.
\end{theorem}
\begin{proof}
Applying the (omitted) analog of Theorem \ref{thm:optimal} for the terminal payoff functional $\fundterm$ of \er{eq:fundterm}, the value functional $\op{S}_{t}\, \fundterm$ enjoys the explicit quadratic representation \er{eq:optimal} with $\op{P}$ replaced with $\optilde{P}$ and $z\equiv 0$. That is, the left-hand equality in \er{eq:Wtilde-dual} holds. Proceeding by an analogous argument to that generating \er{eq:convex-evolution}, define $\optilde{K}_{t} \doteq -\alpha\, \optilde{P}(t) - (1 - \alpha) \, \op{M}\in\Sigma(\cX)$ for any $\alpha\in(0,1)$, and note that it satisfies
$\optilde{P}(t) + \optilde{K}_{t} = (1 - \alpha) \, ( \optilde{P}(t) - \op{M} )> 0$ and $-\optilde{K}_{t} - \op{M} = \alpha \, ( \optilde{P}(t) - \op{M} ) > 0$ for all $t\in(0,\tau^*]$, where the inequalities follow by Theorem \ref{thm:optilde-P-existence}. That is, $\op{S}_{t}\, \fundterm \in \semiconvex{\optilde{K}_{t}}(\cX)$ for all $t\in(0,\tau^*]$. As $\op{S}_{t}\, \fundterm$ is a quadratic functional (as indicated above), it is closed via Lemma \ref{lem:quadratic}. Consequently, the semiconvex duals of both $\fundterm$ and $\op{S}_{t}\, \fundterm$ are well-defined by Theorem \ref{thm:dual}. Set $\tilde a \doteq \op{D}_{\psi}\, \fundterm$. Define $I_t:\cX\times\Ltwo([0,\tau^*];\cW)\mapsinto\R$ by
$$
	I_{t}(x;w)
	\doteq \int_{0}^{t} \demi \langle \xi(s),\, \op{C}\, \xi(s) \rangle - \demi \|w(s)\|^{2} \, ds 
$$
where $\xi(\cdot)$ is the mild solution \er{eq:mild-dynamics} of the abstract Cauchy problem \er{eq:dynamics} satisfying $\xi(0) = x\in\cX$.
Using max-plus integral notation, \er{eq:op-S}, \er{eq:op-D-inverse} and the definition of $\tilde a$ imply that
\begin{align}
	& ( \op{S}_{t} \, \fundterm )(x)
	= 
	\int_{\Ltwo([0,t];\cW)}^{\oplus} I_{t}(x;w) \otimes \fundterm(\xi(t))  \, dw
	= 
	\int_{\Ltwo([0,t];\cW)}^{\oplus} 
	I_{t}(x;w) \otimes  
	( \op{D}_{\psi}^{-1}\, \tilde a ) (\xi(t)) \, dw
	\nn\\
	& =
	\int_{\Ltwo([0,t];\cW)}^{\oplus}
	I_{t}(x;w) \otimes
	\int_{\cX}^{\oplus} \psi(\xi(t),z) \otimes \tilde a(z) \, dz \, dw
	\nn\\
	& = 
	\int_{\cX}^{\oplus}
	\int_{\Ltwo([0,t];\cW)}^{\oplus}
	I_{t}(x;w) \otimes \psi(\xi(t),z) \, dw \otimes \tilde a(z) \, dz
	= 
	\int_{\cX}^{\oplus}
	(\op{S}_{t}\, \psi(\cdot,z))(x) \otimes \tilde a(z) \, dz\,,
	\label{eq:mp-integral-1}
\end{align}
where the second last equality follows by swapping the order of the max-plus integrals (i.e. suprema), while the last equation follows by applying definition \er{eq:op-S} of $\op{S}_{t}$ with terminal payoff $\psi(\cdot,z)$ of \er{eq:terminal}. Subsequently applying \er{eq:op-D-inverse} and \er{eq:primal-W}, swapping the order of the max-plus integrals, and applying \er{eq:op-B} yields
\begin{align}
	& \int_{\cX}^{\oplus}
	(\op{S}_{t}\, \psi(\cdot,z))(x) \otimes \tilde a(z) \, dz
	= \int_{\cX}^{\oplus} (\op{D}_{\psi}^{-1} \, B_{t}(\cdot,z))(x) \otimes \tilde a(z) \, dz
	\nn\\
	& = \int_{\cX}^{\oplus} \int_{\cX}^{\oplus} \psi(x,y) \otimes B_{t}(y,z) \, dy \otimes \tilde a(z) \, dz
	= \int_{\cX}^{\oplus} \psi(x,y) \otimes \int_{\cX}^{\oplus} B_{t}(y,z) \otimes \tilde a(z) \, dz \, dy
	\nn\\
	& = \int_{\cX}^{\oplus} \psi(x,y) \otimes [ \opBmp{t}\, \tilde a ](y) \, dy
	= (\op{D}_{\psi}^{-1} \, \opBmp{t}\, \tilde a)(x)\,.
	\label{eq:mp-integral-2}
\end{align}
Hence, combining the \er{eq:mp-integral-1} and \er{eq:mp-integral-2} and recalling the definition of $\tilde a$ yields
\begin{align}
	( \op{S}_{t} \, \fundterm )(x)
	& = (\op{D}_{\psi}^{-1} \, \opBmp{t}\, \op{D}_{\psi}\, \fundterm)(x)
	\nn
\end{align}
for all $x\in\cX$, as required by the right-hand equality in \er{eq:Wtilde-dual}.
\end{proof}

The dual-space representation provided by Theorem \ref{thm:fund-general} allows the general solution $\optilde{P}$ of the operator differential Riccati equation \er{eq:op-P}, defined with respect to any initialization $\optilde{M}\in\Sigma_{\op{M}}(\cX)$ as per Theorem \ref{thm:optilde-P-existence}, to be represented in terms of the max-plus integral operator $\opBmp{t}$ of \er{eq:op-B}. However, $\opBmp{t}$ is defined via the max-plus kernel $B_{t}$ of \er{eq:dual-W}, \er{eq:B-explicit}, which is itself derived from the particular solution $\op{P}$ of the operator differential Riccati equation \er{eq:op-P} that is defined with respect to the initialization $\op{M}\in\Sigma(\cX)$ as per Assumption \ref{ass:coercive}. That is, Theorem \ref{thm:fund-general} provides a representation for the general solution of the operator differential Riccati equation \er{eq:op-P} in terms of a particular solution of the same equation. It implies the following commutation diagram:
\begin{align}
&
\begin{CD}
\optilde{M}	@>\text{\er{eq:fundterm}}>>	\fundterm	@>\op{S}_{t}>>		\op{S}_{t}\, \fundterm		@>\text{\er{eq:Wtilde-dual}}>>
																							\optilde{P}(t)	\\
			@.								@VV\op{D}_{\psi}V						@AA\op{D}_{\psi}^{-1}A		\\
			@. 			\op{D}_{\psi}\, \fundterm	@>\opBmp{t}>>	\opBmp{t}\, \op{D}_{\psi}\, \fundterm 
																				@.
\end{CD}
\label{eq:commute-diagram}
\end{align}
In subsequent applications of \er{eq:commute-diagram}, an explicit evaluation of $\op{D}_{\psi} \, \fundterm$ and $\op{D}_{\psi}\, \op{S}_{t}\, \fundterm$ is useful. These evaluations are provided in the following two lemmas.
\begin{lemma}
\label{lem:fundterm-dual}
Given the semiconvex dual operator $\op{D}_\psi$ of \er{eq:op-D}, the semiconvex dual $\op{D}_{\psi}\, \fundterm$ of the terminal payoff $\fundterm$ of \er{eq:fundterm} is given for all $z\in\cX$ by
\begin{align}
	(\op{D}_{\psi}\, \fundterm)(z)
	& = - \demi \, \langle z,\, \opNtilde\, z \rangle
	\label{eq:fundterm-dual}
\end{align}
where $\opNtilde\in\Sigma(\cX)$ is defined by
\begin{align}
	\opNtilde
	& \doteq \op{M} + \op{M} \, ( \opMtilde - \op{M} )^{-1} \, \op{M}
	= \op{M}\, (\opMtilde - \op{M})^{-1}\, \opMtilde\,,
	\label{eq:op-N-tilde}
\end{align}
with $\op{M}$, $\opMtilde$ as per Assumption \ref{ass:coercive} and Theorem \ref{thm:optilde-P-existence}, and satisfies $\op{D}_{\psi}\, \fundterm\in\semiconcave{\op{K}}(\cX)$ for any $\op{K}\in\Sigma(\cX)$ satisfying $\op{K}+\op{M} > 0$.
\end{lemma}
\begin{proof}
Applying \er{eq:op-D} to \er{eq:fundterm}, $(\op{D}_{\psi}\, \fundterm)(z) = - \int_{\cX}^{\oplus} \pi^{z}(x)\, dx$, where
\begin{align}
	\pi^{z}(x)
	& \doteq  
	\demi \langle x - z,\, \op{M}\, (x-z) \rangle
	- \demi \langle x,\, \opMtilde\, x \rangle
	\nn\\
	& =
	\demi \langle x,\, (\op{M} - \opMtilde) \, x \rangle + 
	\langle x,\, -\op{M}\, z \rangle + 
	\demi \langle z,\, \op{M}\, z \rangle
	\nn\\
	& 
	= b(x) + \demi \langle z,\, \op{M}\, z \rangle\,,
	\nn
\end{align}
where $b(x) \doteq \demi \langle x,\, (\op{M} - \opMtilde) \, x \rangle + \langle x,\, -\op{M}\, z \rangle$. That is,
\begin{align}
	(\op{D}_{\psi}\, \fundterm)(z)
	& = - \demi \langle z,\, \op{M}\, z \rangle - \int_{\cX}^{\oplus} b(x)\, dx\,.
	\label{eq:dual-fund-b}
\end{align}
Here, $\op{M} - \opMtilde\in\Sigma(\cX)$, while $-\op{M}\, z\in\cX$. Furthermore,  $\opMtilde>\op{M}$, as $\optilde{M}\in\Sigma_{\op{M}}(\cX)$ as per Theorem \ref{thm:optilde-P-existence}. Hence, $\sup_{x\in\cX} b(x) < \infty$, with Lemma \ref{lem:pseudo} requiring that the supremum in \er{eq:dual-fund-b} be attained at $x^{*} = -(\opMtilde - \op{M})^{-1} \, \op{M}\, z$, where invertibility of $\opMtilde - \op{M}$ follows by the coercivity condition $\optilde{M}\in\Sigma_{\op{M}}(\cX)$ specified in Theorem \ref{thm:optilde-P-existence}. Consequently, $\int_{\cX}^{\oplus} b(x)\, dx = b(x^{*}) = \demi \langle \op{M}\, z,\, (\opMtilde - \op{M})^{-1}\, \op{M}\, z \rangle$, so that \er{eq:dual-fund-b} implies that \er{eq:fundterm-dual} holds with $\opNtilde$ given by the left-hand equality in \er{eq:op-N-tilde}. Furthermore,
$
	\opNtilde = \op{M} + \op{M}\, (\opMtilde - \op{M})^{-1}\, \op{M}
	= \op{M} - \op{M} \, (\opMtilde - \op{M})^{-1}\, [ (\opMtilde - \op{M}) - \opMtilde]
	= \op{M} - \op{M} + \op{M} \, ( \opMtilde - \op{M})^{-1}\, \opMtilde 
$,
as per \er{eq:op-N-tilde}. Selecting any $\op{K}\in\Sigma(\cX)$ such that $\op{K} + \op{M} > 0$, and applying \er{eq:fundterm-dual}, 
$
	- [ (\op{D}_{\psi}\, \fundterm)(z) - \demi \langle z,\, \op{K} \, z \rangle ]
	= \demi \langle z,\, \op{M}\, (\opMtilde - \op{M})^{-1}\, \op{M}\, z \rangle
	+ \demi \langle z,\, ( \op{K} + \op{M} ) \, z \rangle
$
for all $z\in\cX$. By inspection, this functional is positive, and hence convex by Lemma \ref{lem:quadratic}. That is, $(\op{D}_{\psi}\, \fundterm)(z) - \demi \langle z,\, \op{K} \, z \rangle$ defines a concave functional, so that $\op{D}_{\psi}\, \fundterm\in\semiconcave{\op{K}}(\cX)$ by \er{eq:semiconcave}.
\end{proof}

\begin{lemma}
\label{lem:op-S-fundterm-dual}
Given the semiconvex dual operator $\op{D}_\psi$ of \er{eq:op-D}, the semiconvex dual $\op{D}_{\psi}\, \op{S}_{t}\, \fundterm$ of the value functional $\op{S}_{t}\, \fundterm$ of \er{eq:Wtilde-dual} corresponding to the terminal payoff $\fundterm$ of \er{eq:fundterm} is given for all $z\in\cX$ and $t\in(0,\tau^*]$ by
\begin{align}
	(\op{D}_{\psi}\, \op{S}_{t}\, \fundterm)(z)
	& = -\demi\, \langle z,\, \optilde{N}_{t} \, z \rangle
	\label{eq:op-S-fundterm-dual}
\end{align}
where $\optilde{N}_{t}\in\Sigma(\cX)$ is defined by
\begin{align}
	\optilde{N}_{t}
	& \doteq \op{M} + \op{M} \, ( \opPtilde(t) - \op{M} )^{-1}\, \op{M} 
	= \op{M} \, ( \opPtilde(t) - \op{M} )^{-1} \, \opPtilde(t)\,,
	\label{eq:op-N-tilde-t}
\end{align}
and $\op{M}\in\Sigma(\cX)$ is as per Assumption \ref{ass:coercive}.
\end{lemma}
\begin{proof}
As $\opPtilde(t) > \op{M}$ by Theorem \ref{thm:optilde-P-existence}, an analogous argument to that yielding \er{eq:convex-evolution} follows, with $\op{P}(t)$ replaced with $\opPtilde(t)$. Consequently, there exists a $\optilde{K}_{t}\in\Sigma(\cX)$ such that $\op{S}_{t}\, \fundterm\in\semiconvex{\optilde{K}_{t}}(\cX)$. Similarly, assertion (i) of Lemma \ref{lem:quadratic} and \er{eq:Wtilde-dual} imply that $\op{S}_{t}\, \fundterm$ is closed. Hence, the semiconvex dual $\op{D}_{\psi}\, \op{S}_{t}\, \fundterm$ is well-defined by Theorem \ref{thm:dual}. So, applying \er{eq:op-D} to \er{eq:Wtilde-dual} in an analogous fashion to the proof of Lemma \ref{lem:fundterm-dual} (i.e. replacing $\opMtilde$ with $\opPtilde(t)$ and noting that $\opPtilde(t) - \op{M}$ is coercive and hence invertible) yields \er{eq:op-S-fundterm-dual}.
\end{proof}

Theorem \ref{thm:fund-general} states that the value functional $\op{S}_{t}\, \fundterm$ of \er{eq:Wtilde-dual} may be identified with the  solution $\optilde{P}$ of the operator differential Riccati equation \er{eq:op-P} satisfying $\optilde{P}(0) = \optilde{M}$ for any $\optilde{M} \in\Sigma_{\op{M}}(\cX)$. Hence, $\optilde{P}$ may be propagated to longer time horizons via the dynamic programming evolution operator $\op{S}_t$ of \er{eq:op-S}. Furthermore, Theorem \ref{thm:fund-general} also states that this propagation can be represented in a max-plus dual space via the max-plus integral operator $\opBmp{t}$ of \er{eq:op-B}. Consequently, as $\op{S}_t$ satisfies the semigroup property \er{eq:op-S-semigroup}, it follows that $\opBmp{t}$ of \er{eq:op-B} inherits an analogous semigroup property that may be used to propagate $\optilde{P}$. In particular, the set and binary operation pair
\begin{align}
	\sigma_{\tau^*}(\op{B}^\oplus)
	& \doteq (\Gamma_{\tau^*}(\op{B}^\oplus), \, \circ_{\tau^*})
	\label{eq:dual-space-semigroup}
\end{align}
defined as per \er{eq:sigma-semigroup} can be regarded as the {\em max-plus dual space fundamental solution semigroup} for the operator differential Riccati equation \er{eq:op-P} on the interval $(0,\tau^*]$. This is formalized by the following theorem and corollary.
\begin{theorem}
\label{thm:propagate-fund}
The max-plus integral operator $\opBmp{t}$ of \er{eq:op-B} satisfies the semigroup property
\begin{align}
	\opBmp{\tau + t}\, \tilde a
	& = \opBmp{\tau}\, \opBmp{t} \, \tilde a
	\label{eq:propagate-fund}
\end{align}
for all $\tau,t\in(0,\tau^*]$, $t+\tau\in(0,\tau^*]$, where $\tilde a\doteq \op{D}_\psi\, \fundterm$ is defined via \er{eq:op-D} and \er{eq:fundterm}.
\end{theorem}
\begin{proof}
Applying the semigroup property \er{eq:op-S-semigroup} and Theorem \ref{thm:fund-general},
\begin{align}
	\op{D}_{\psi}^{-1} \, \opBmp{\tau+t}\, \op{D}_{\psi}\, \fundterm
	& = \op{S}_{\tau+t} \, \fundterm
	\nn\\
	& = \op{S}_{\tau} \, \op{S}_{t} \, \fundterm
	= \op{D}_{\psi}^{-1} \, \opBmp{\tau} \, \op{D}_{\psi}\, \op{D}_{\psi}^{-1} \, \opBmp{t} \, \op{D}_{\psi}\, \fundterm 
	=  \op{D}_{\psi}^{-1} \, \opBmp{\tau}\, \opBmp{t} \, \op{D}_{\psi}\, \fundterm\,,
	\nn
\end{align}
That is, $\op{D}_{\psi}^{-1}\, \opBmp{\tau+t} \, \tilde a = \op{D}_{\psi}^{-1}\, \opBmp{\tau} \, \opBmp{t} \, \tilde a$, where $\tilde a \doteq \op{D}_{\psi} \, \fundterm$. Applying the semiconvex dual operator $\op{D}_{\psi}$ of \er{eq:op-D} to both sides thus yields the semigroup property \er{eq:propagate-fund}.
\end{proof}

\begin{corollary} 
\label{cor:propagate-fund}
Given $\tilde a\doteq\op{D}_\psi\, \fundterm$ and the max-plus integral operators $\op{B}_\tau^\oplus$, $\op{B}_t^\oplus$ defined by \er{eq:op-B} for $\tau,t\in(0,\tau^*]$, $\tau+t\in(0,\tau^*]$, 
\begin{align}
	\opBmp{\tau}\, \opBmp{t} \, \tilde a = (\opBmp{\tau}\, \opBmp{t} \, \tilde a)(\cdot)
	& = \int_{\cX}^{\oplus} B_{\tau,t}(\cdot,z) \otimes \tilde a(z)\, dz
	\label{eq:op-B-B}
\end{align}
in which kernel $B_{\tau,t}:\cX\times\cX\mapsinto\R^-$ is defined by
\begin{align}
	B_{\tau,t}(y,z)
	& \doteq \demi \, \langle y,\, \betaOO{\tau,t}\, y \rangle + 
	\langle z,\, \betaOT{\tau,t}\, y \rangle + 
	\demi\, \langle z,\, \betaTT{\tau,t}\, z \rangle \,,
	\label{eq:op-B-B-kernel}
\end{align}
with $\betaOO{\tau,t},\, \betaTT{\tau,t}\in\Sigma(\cX)$, $\betaOT{\tau,t}\in\bo(\cX)$, defined with respect to $\betaOO{\star},\, \betaTT{\star}\in\Sigma(\cX)$, 
$\betaOT{\star} \in\bo(\cX)$, $\star\in\{t,\tau\}$, of \er{eq:op-B-11}, \er{eq:op-B-12}, \er{eq:op-B-22}, by
\begin{align}
	\betaOO{\tau, t}
	& \doteq
	\betaOO{\tau} - (\betaOT{\tau})' \, \left(\betaTT{\tau} + \betaOO{t} \right)^{+} \, \betaOT{\tau}\,,
	\label{eq:prop-11}
	\\
	\betaOT{\tau, t}
	& \doteq - \betaOT{t} \, \left( \betaTT{\tau} + \betaOO{t} \right)^{+} \, \betaOT{\tau}\,,
	\label{eq:prop-12}
	\\
	\betaTT{\tau, t}
	& \doteq \betaTT{t} - \betaOT{t} \, \left( \betaTT{\tau} + \betaOO{t} \right)^{+} \, (\betaOT{t})'\,,
	\label{eq:prop-22}
\end{align}
in which $(\cdot)^+$ denotes the Moore-Penrose pseudo-inverse (see Lemma \ref{lem:pseudo}).
\end{corollary}
\begin{proof}
Fix $\tilde a\doteq\op{D}_\psi\, \fundterm$ and $\tau,t\in(0,\tau^*]$, $\tau+t\in(0,\tau^*]$. Applying definition \er{eq:op-B},
\begin{align}
	\opBmp{\tau} \, \opBmp{t} \, \tilde a
	& =
	\int_{\cX}^{\oplus} B_{\tau}(\cdot,\xi) \otimes \left[ \int_{\cX}^{\oplus} B_{t}(\xi,z) \otimes \tilde a(z) \, dz \right] \, d\xi
	\nn\\
	& =
	\int_{\cX}^{\oplus} \left[ \int_{\cX}^{\oplus}  B_{\tau}(\cdot,\xi) \otimes B_{t}(\xi,z) \, d\xi \right] \otimes \tilde a(z) \, dz\,.
	\label{eq:op-B-B-form}
\end{align}
That is, \er{eq:op-B-B} holds with $B_{\tau,t}(y,z)\doteq\int_{\cX}^{\oplus}  B_{\tau}(y,\xi) \otimes B_{t}(\xi,z) \, d\xi$. Furthermore, Lemma \ref{lem:B-explicit} states that $B_{\tau}$ and $B_{t}$ are quadratic functionals with the explicit form \er{eq:B-explicit}, implying that the functional $\pi_{\tau,t}^{y,z}(\xi) \doteq B_{\tau}(y,\xi) \otimes B_{t}(\xi,z)$ is also quadratic. That is,
\begin{align}
	\pi_{\tau,t}^{y,z}(\xi)
	& =
	\demi\, \langle y,\, \betaOO{\tau}\, y\rangle + \langle \xi,\, \betaOT{\tau} \, y \rangle +
	\demi\, \langle \xi,\, \betaTT{\tau} \, \xi \rangle 
	\nn\\
	& \quad\quad
	+ \demi\, \langle \xi,\, \betaOO{t}\, \xi\rangle + \langle z,\, \betaOT{t} \, \xi \rangle +
	\demi\, \langle z,\, \betaTT{t} \, z \rangle
	\nn\\
	& = b(\xi) + \demi\, \langle y,\, \betaOO{\tau}\, y\rangle + 
	\demi\, \langle z,\, \betaTT{t} \, z \rangle\,,
	\nn
\end{align}
where 
\begin{align}
	b(\xi) 
	& \doteq \demi\, \langle \xi, \, (\betaTT{\tau} + \betaOO{t})\, \xi \rangle + 
	\langle \xi,\, \betaOT{\tau} \, y + (\betaOT{t})'\, z \rangle\,.
	\label{eq:op-B-B-b}
\end{align}
That is,
\begin{align}
	B_{\tau,t}(y,z)
	& = \demi\, \langle y,\, \betaOO{\tau}\, y\rangle + 
	\demi\, \langle z,\, \betaTT{t} \, z \rangle + \int_{\cX}^{\oplus} b(\xi) \, d\xi\,.
	\label{eq:B-B-integral}
\end{align}
In order to derive the form \er{eq:op-B-B-kernel} for $B_{\tau,t}$, the supremum on the right-hand side of \er{eq:B-B-integral} must be shown to be finite, and subsequently evaluated. To do this, first note that $B_{\tau,t}$ as written in \er{eq:B-B-integral} is defined entirely in terms of the unique solution $\optilde{P}$ of the operator differential Riccati equation \er{eq:op-P} initialized with $\optilde{P}(0) = \optilde{M}\in\Sigma_\op{M}(\cX)$ as per Theorem \ref{thm:optilde-P-existence}. As this solution exists on the interval $[0,\tau^*]$, Theorem \ref{thm:fund-general} and Lemma \ref{lem:op-S-fundterm-dual} imply that $\op{S}_{\tau+t}\, \fundterm$ and $\op{D}_{\psi}\, \op{S}_{\tau+t}\, \fundterm = \opBmp{\tau+t}\, \op{D}_{\psi}\, \fundterm$ are well-defined quadratic functionals (i.e. finite-valued everywhere on $\cX$) for all $\tau,t\in(0,\tau^*]$, $\tau+t\in(0,\tau^*]$. Indeed, these functionals are given explicitly by \er{eq:Wtilde-dual} and \er{eq:op-S-fundterm-dual}, with
\begin{align}
	(\op{S}_{\tau+t}\, \fundterm)(x)
	& = \demi \, \langle x,\, \opPtilde(\tau+t)\, x\rangle\,,
	\quad
	(\op{D}_{\psi}\, \op{S}_{\tau+t}\, \fundterm)(z)
	= -\demi \, \langle x,\, \optilde{N}_{\tau+t}\, z \rangle\,,
	\nn
\end{align}
where $\opPtilde(\tau+t)$ is as per Theorem \ref{thm:optilde-P-existence} and $\opNtilde_{\tau+t}$ is as per \er{eq:op-N-tilde-t}. So, Theorem \ref{thm:propagate-fund} states that $\opBmp{\tau}\, \opBmp{t}\, \tilde a = \opBmp{\tau+t}\, \tilde a = \op{D}_{\psi}\, \op{S}_{\tau+t}\, \fundterm$ is a well-defined quadratic functional, where $\tilde a \doteq \op{D}_{\psi}\, \fundterm$. That is, $\tilde a, (\opBmp{\tau}\, \opBmp{t})\, \tilde a:\cX\mapsinto\R$ are finite-valued everywhere on $\cX$. As $\opBmp{\tau}\, \opBmp{t}$ of \er{eq:op-B-B} is a max-plus integral operator of the form \er{eq:op-O} as per \er{eq:op-B-B}, Lemma \ref{lem:op-O-finite} implies that the kernel $B_{\tau,t}$ of $\opBmp{\tau}\, \opBmp{t}$ defined via \er{eq:op-B-B} and \er{eq:B-B-integral} must be finite-valued. Hence, the integral on the right-hand side of \er{eq:B-B-integral} must be finite. Recalling the definition \er{eq:op-B-B-b} of the quadratic functional $b:\cX\mapsinto\R$, observe that $\betaTT{\tau} + \betaOO{t}$ is self-adjoint by \er{eq:op-B-11} and \er{eq:op-B-22}, while $\betaOT{\tau}\, y + (\betaOT{t})'\, z\in\cX$. Hence, applying Lemma \ref{lem:pseudo}, the pseudo-inverse of $\betaTT{\tau} + \betaOO{t}$ must exist, with the supremum in \er{eq:B-B-integral} attained at
$\xi^{*} \doteq - ( \betaTT{\tau} + \betaOO{t} )^{+}\, (\betaOT{\tau} \, y + (\betaOT{t})'\, z)$. That is,
\begin{align}
	\int_{\cX}^{\oplus} b(\xi) \, d\xi
	& = b(\xi^{*})
	 = - \demi \left\langle \betaOT{\tau} \, y + (\betaOT{t})'\, z,\, 
	 \left( \betaTT{\tau} + \betaOO{t} \right)^{+}\, (\betaOT{\tau} \, y + (\betaOT{t})'\, z) \right\rangle\,.
	 \nn
\end{align}
Substituting this into \er{eq:B-B-integral} yields
\begin{align}
	B_{\tau,t}(y,z)
	& = \demi \left\langle y,\, \left[ \betaOO{\tau} 
	-(\betaOT{\tau})' \left( \betaTT{\tau} + \betaOO{t} \right)^{+}\, \betaOT{\tau} \right] y \right\rangle
	\nn\\
	& 
	\quad\quad 
	+ \left\langle z,\,
	\left[ - \betaOT{t} \left( \betaTT{\tau} + \betaOO{t} \right)^{+}\, \betaOT{\tau} \right] y \right\rangle
	\nn\\
	& 
	\quad\quad 
	+ \demi \left\langle z,\, 
	\left[ \betaTT{t} - \betaOT{t} \left( \betaTT{\tau} + \betaOO{t} \right)^{+} \, (\betaOT{t})' \right] z \right\rangle\,,
	\nn
\end{align}
which is as per \er{eq:op-B-B-kernel} via the operator definitions \er{eq:prop-11}, \er{eq:prop-12}, and \er{eq:prop-22}.
\end{proof}

The specific details of how the max-plus dual space fundamental solution semigroup \er{eq:dual-space-semigroup} may be applied to evaluate the solution $\optilde{P}(t)$, $t\in(0,\tau^*]$, of \er{eq:op-P} initialized with $\optilde{P}(0) = \optilde{M}\in\Sigma_\op{M}(\cX)$ follow in the next section.


\section{Solving the operator differential Riccati equation \er{eq:op-P}}
\label{sec:solve}
The remaining objective is to illustrate how solutions $\optilde{P}\in C_0([0,t];\Sigma_\op{M}(\cX))$ of the operator differential Riccati equation \er{eq:op-P} can be evaluated at some time $t\in(0,\tau^*]$ for any initialization $\optilde{P}(0) = \optilde{M}\in\Sigma_\op{M}(\cX)$ using the max-plus dual space fundamental solution semigroup \er{eq:dual-space-semigroup}. Three main steps are involved. First, the max-plus integral operator $\opBmp{\delta}$ is obtained for some incremental intermediate time 
\begin{align}
	\delta\doteq t / \kappa\in(0,\tau^*]\,, \quad \kappa\in\Z_{>0}\,,
	\label{eq:delta}
\end{align}
from the unique solution $\op{P}\in C_0([0,\tau^*];\Sigma(\cX))$ initialized with $\op{P}(0) = \op{M}\in\Sigma(\cX)$ as per Assumption \ref{ass:coercive}. Second, $\opBmp{t}$ is derived from $\opBmp{\tau}$ via the fundamental solution semigroup property of Theorem \ref{thm:propagate-fund} and Corollary \ref{cor:propagate-fund}. Third, $\opBmp{t}$ is applied to evaluate the solution $\opPtilde$ initialized with $\optilde{P}(0) = \opMtilde\in\Sigma_{\op{M}}(\cX)$ at time $t\in(0,\tau^*]$.

It is important to note that where solutions $\optilde{P}$ of the operator differential Riccati equation \er{eq:op-P} satisfying $\optilde{P}(0) = \opMtilde$ are to be evaluated at the same time $t\in(0,\tau^*]$ for a collection of initializations $\opMtilde\in\Sigma_{\op{M}}(\cX)$, the first two steps need only be performed once, whereupon step three can be repeated for each initialization $\optilde{M}$ of interest.

In the following three subsections, these three main steps are described separately. A summarized recipe for computing $\optilde{P}(t)$ from $\op{P}(\delta)$ is also provided. For reasons of brevity, an error analysis for this recipe is not included.


\subsection{Step 1 -- Obtaining $\opBmp{\delta}$ from $\op{P}({\delta})$} With $t\in(0,\tau^*]$ fixed, select $\kappa\in\Z_{>1}$ and define $\delta\in(0,\tau^*]$ as per \er{eq:delta}. Recall that the max-plus integral operator $\opBmp{\delta}$ of \er{eq:op-B} is defined via kernel $B_{\delta}$. Lemma \ref{lem:B-explicit} provides an explicit quadratic functional representation \er{eq:B-explicit} for $B_{\delta}$ in terms of $\betaOO{\delta},\, \betaTT{\delta}\in\Sigma(\cX)$, $\betaOT{\delta}\in\bo(\cX)$ of \er{eq:op-B-11}, \er{eq:op-B-12}, \er{eq:op-B-22}, with
\begin{align}
	\betaOO{\delta}
	& \doteq - \op{M} - \op{M} \left(\op{P}(\delta) - \op{M} \right)^{-1} \op{M}\,,
	\ooer{eq:op-B-11}
	\\
	\betaOT{\delta}
	& \doteq - \op{Q}(\delta)' \left( \op{P}(\delta) - \op{M} \right)^{-1}\, \op{M}\,,
	\ooer{eq:op-B-12}
	\\
	\betaTT{\delta}
	& \doteq - \op{Q}(\delta)' \left( \op{P}(\delta) - \op{M} \right)^{-1} \op{Q}(\delta) + \op{R}(\delta)\,.
	\ooer{eq:op-B-22}
\end{align}
(Note that operators $\op{Q}(\delta)\in\bo(\cX)$ and $\op{R}(\delta)\in\Sigma(\cX)$ are uniquely determined by $\op{P}\in C_0([0,\tau^*];\Sigma(\cX))$ via Theorem \ref{thm:existence-uniqueness-Q-R}.) These operators completely describe the max-plus integral operator $\opBmp{\tau}$ via \er{eq:B-explicit} and \er{eq:op-B}.


\subsection{Step 2 -- Obtaining $\opBmp{t}$ from $\opBmp{\delta}$}
\label{sec:step-2}
The max-plus dual space fundamental solution semigroup \er{eq:dual-space-semigroup} of Theorem \ref{thm:propagate-fund} and Corollary \ref{cor:propagate-fund} provides a iterative mechanism for constructing $\opBmp{t}$ from $\opBmp{\delta}$. With a view to evaluating $\opBmp{t}\, \tilde a = \opBmp{\delta} \, \opBmp{\delta} \, \cdots \, \opBmp{\delta}\, \tilde a$ ($\kappa$ times) given $\tilde a \doteq \op{D}_{\psi}\, \fundterm$, select $t = (k-1)\, \delta$ with $\tau=\delta$ in \er{eq:prop-11}, \er{eq:prop-12}, \er{eq:prop-22}. This yields a linear iteration of the triple $(\betaOOhat{k},\, \betaOThat{k},\, \betaTThat{k})\in\bo(\cX)^3$ given by
\begin{align}
	\betaOOhat{k}
	& \doteq
	\betaOO{\delta} - (\betaOT{\delta})' \, \left(\betaTT{\delta} + \betaOOhat{k-1} \right)^{+} \, \betaOT{\delta}\,,
	\label{eq:lin-iter-11}
	\\
	\betaOThat{k}
	& \doteq - \betaOThat{k-1} \, \left( \betaTT{\delta} + \betaOOhat{k-1} \right)^{+} \, \betaOT{\delta}\,,
	\label{eq:lin-iter-12}
	\\
	\betaTThat{k}
	& \doteq \betaTThat{k-1} - \betaOThat{k-1} \, \left( \betaTT{\delta} + \betaOOhat{k-1} \right)^{+} \, (\betaOThat{k-1})'\,,
	\label{eq:lin-iter-22}
\end{align}
for $k=2,\cdots, \kappa$, initialized with 
\begin{align}
	(\betaOOhat{1},\, \betaOThat{1},\, \betaTThat{1})
	& = (\betaOO{\delta},\, \betaOT{\delta},\, \betaTT{\delta})
	\label{eq:lin-iter-IC}
\end{align}
via \er{eq:op-B-11}, \er{eq:op-B-12} and \er{eq:op-B-22}. These operators completely describe (via \er{eq:op-B-B} and \er{eq:op-B-B-kernel}) the max-plus integral operator $\opBmp{\delta}\, \opBmp{(k-1)\, \delta}$ for $k=2,\cdots \kappa$. The desired max-plus integral operator $\opBmp{t}$ follows from the $k=\kappa^{\text{th}}$ iterate
\begin{align}
	(\betaOO{t},\, \betaOT{t},\, \betaTT{t})
	& = (\betaOOhat{\kappa},\, \betaOThat{\kappa},\, \betaTThat{\kappa})\,.
	\label{eq:lin-iter-final}
\end{align}
For all iterates, Theorem \ref{thm:fund-general} states that $(\op{S}_{k\, \delta}\, \fundterm)(x) = (\op{D}_{\psi}^{-1}\, \opBmp{\delta}\, \opBmp{(k-1)\, \tau} \, \tilde a)(x) = \demi \, \langle x,\, \optilde{P}(k\, \delta)\, x \rangle$, where $\tilde a\doteq \op{D}_{\psi}\, \fundterm$. Hence, applying an argument analogous to that used in the proof of Corollary \ref{cor:propagate-fund}, existence of $\optilde{P}(k\,\delta)$ for each $k=1,\cdots,\kappa$ as provided by Assumption \ref{ass:coercive} guarantees that the operator iteration \er{eq:lin-iter-11}, \er{eq:lin-iter-12}, \er{eq:lin-iter-22} remains well-defined, subject to the aforementioned initialization \er{eq:lin-iter-IC}. In particular, the pseudo-inverses employed there must exist, and the operators must remain bounded and linear.

Other iterations are also possible. For example, as an alternative to a linear iteration, a time-step doubling iteration may be used to construct $\opBmp{t}$. Such a scheme requires fewer iterations (than the linear scheme illustrated above) to reach $t\in(0,\tau^*]$. The details of such iterations are omitted for brevity.


\subsection{Step 3 -- Obtaining $\optilde{P}(t)$ from $\opBmp{t}$ and $\opMtilde$}
Theorem \ref{thm:fund-general} and the commutation diagram \er{eq:commute-diagram} provide the mechanism for evaluating the solution $\optilde{P}$ of the operator differential Riccati equation \er{eq:op-P} satisfying $\optilde{P}(0) = \optilde{M}$ at time $t\in(0,\tau^*]$ for any $\optilde{M}\in\Sigma_{\op{M}}(\cX)$ via the max-plus integral operator $\opBmp{t}$ obtained in the previous step. In particular, \er{eq:Wtilde-dual} states that
\begin{align}
	\demi\, \langle x,\, \optilde{P}(t)\, x \rangle
	& = ( \op{D}_{\psi}^{-1}\, \opBmp{t}\, \tilde a)(x)\,,
	\quad\quad	
	\tilde a \doteq \op{D}_{\psi}\, \fundterm\,.
	\label{eq:step-3-a}
\end{align}
Recall that Lemma \ref{lem:fundterm-dual}, $\tilde a = \op{D}_{\psi}\, \fundterm = -\demi\, \langle z,\, \optilde{N}\, z \rangle$, with $\optilde{N}\in\Sigma(\cX)$ is as per \er{eq:op-N-tilde}. Applying $\opBmp{t}$ of the previous step to $\tilde a$ yields
\begin{align}
	(\opBmp{t}\, \tilde a)(y) 
	& = 
	\int_{\cX}^{\oplus} B_{t}(y,z)\otimes \tilde a(z)\, dz
	\nn\\
	& =
	\demi \, \langle y,\, \betaOO{t}\, y\rangle +
	\int_{\cX}^{\oplus} 
	\demi\, \langle z,\, (\betaTT{t} - \optilde{N}) \, z \rangle + 
	\langle z,\, \betaOT{t}\, y \rangle
	\, dz
	\nn
\end{align}
for all $y\in\cX$. An analogous argument to the proof of Corollary \ref{cor:propagate-fund} implies the existence of a finite right-hand side supremum. 
%
%
Consequently, Lemma \ref{lem:pseudo} implies that a bounded linear operator $\optilde{T}_{t}$ exists such that
\begin{align}
	(\opBmp{t}\, \tilde a)(y) 
	& = 
	\demi\, \langle y,\, \optilde{T}_{t}\, y \rangle\,,
	\quad\quad
	\optilde{T}_{t}
	\doteq \betaOO{t} - (\betaOT{t})' \left( \betaTT{t} - \optilde{N} \right)^{+} 
	\betaOT{t}\,.
	\nn
\end{align}
Applying the inverse semiconvex dual operator $\op{D}_{\psi}^{-1}$ of \er{eq:op-D-inverse} to obtain \er{eq:step-3-a},
\begin{align}
	(\op{D}_{\psi}^{-1}\, \opBmp{t}\, \tilde a)(x)
	& = \int_{\cX}^{\oplus} \psi(x,y) \otimes (\opBmp{t}\, \tilde a)(y) \, dy
	\nn\\
	& = 
	\demi\, \langle x,\, \op{M}\, x \rangle
	+ \int_{\cX}^{\oplus} \demi \, \langle y,\, (\optilde{T}_{t} + \op{M}) \, y \rangle
	+ \langle y,\, -\op{M}\, x \rangle \, dy\,.
	\nn
\end{align}
Again, an analogous argument to the proof of Corollary \ref{cor:propagate-fund} implies the existence of a finite right-hand side supremum. Consequently, Lemma \ref{lem:pseudo} implies that a bounded linear operator $\optilde{O}_{t}$ exists such that
\begin{align}
	(\op{D}_{\psi}^{-1}\, \opBmp{t}\, \tilde a)(x)
	& = 
	\demi \, \langle x,\, \optilde{O}_{t}\, x \rangle\,,
	\quad\quad
	\optilde{O}_{t}
	\doteq
	\op{M} - \op{M} \left( \optilde{T}_{t} + \op{M} \right)^{+} \op{M}\,.
	\label{eq:step-3-c}
\end{align}
Combining \er{eq:step-3-a} and \er{eq:step-3-c} implies that $\demi \, \langle x,\, \optilde{P}(t)\, x \rangle = \demi \, \langle x,\, \optilde{O}_{t}\, x \rangle$ for all $x\in\cX$, or
\begin{align}
	\opPtilde(t)
	& = \op{M} - \op{M} \left( \optilde{T}_{t} + \op{M} \right)^{+} \op{M}\,.
	\label{eq:construct-general}
\end{align}


\newcommand{\done}{\ding{182}}
\newcommand{\dtwo}{\ding{183}}
\newcommand{\dthree}{\ding{184}}
\newcommand{\dfour}{\ding{185}}
\newcommand{\dfive}{\ding{186}}

\subsection{Recipe}
\label{sec:recipe}
The solution $\optilde{P}$ of the operator differential Riccati equation \er{eq:op-P} initialized with $\optilde{P}(0) = \optilde{M}$ for any $\optilde{M}\in\Sigma_{\op{M}}(\cX)$ can be evaluated at any time within an interval of existence $(0,\tau^*]$ using the particular solution $\op{P}$ of the same equation initialized with $\op{P}(0) = \op{M}\in\Sigma(\cX)$ as specified in Assumption \ref{ass:coercive} via the following recipe:

\begin{center}
\vspace{2mm}
\hspace{-10mm}
\parbox[c]{13cm}{\centering
\begin{itemize}
\item[\done]
Select a time $t\in(0,\tau^*]$, $\tau^*\in\R_{>0}$ as per \er{eq:tau-star}, at which evaluation of the solution $\optilde{P}$ of the operator differential Riccati equation \er{eq:op-P} satisfying $\optilde{P}(0) = \optilde{M}$ for some $\optilde{M}\in\Sigma_{\op{M}}(\cX)$ is required. Fix iteration integer $\kappa\in\Z_{>1}$ and time $\delta \doteq t / \kappa$ as per \er{eq:delta}. Construct the bounded linear operators $\betaOO{\tau}$, $\betaOT{\tau}$ and $\betaTT{\tau}$ from the evaluation of the known particular solution $\op{P}(\tau)$ according to \er{eq:op-B-11}, \er{eq:op-B-12} and \er{eq:op-B-22}. 

\item[\dtwo]
Iterate the operator triple $(\betaOOhat{k},\, \betaOThat{k},\, \betaTThat{k})$ as per \er{eq:lin-iter-11}, \er{eq:lin-iter-12}, \er{eq:lin-iter-22} for $k=2, \cdots, \kappa$, subject to the initialization \er{eq:lin-iter-IC}, to obtain the final operator triple $(\betaOO{t},\, \betaOT{t},\, \betaTT{t})$ at time $t$ as per \er{eq:lin-iter-final}.

\item[\dthree] 
Select any initial condition $\opMtilde\in\Sigma_{\op{M}}(\cX)$. Evaluate the solution $\optilde{P}$ of the operator differential Riccati equation satisfying $\optilde{P}(0) = \optilde{M}$ at time $t$ via \er{eq:construct-general}.
\end{itemize}
\vspace{3mm}
}
\end{center}


\subsection{An illustrative example}
\label{sec:example}
A brief example is provided to illustrate an application of the recipe of Section \ref{sec:recipe} to the numerical evaluation of solutions of a specific operator differential Riccati equation of the form \er{eq:op-P}. With $\partial$ denoting differentiation, select 
\begin{align}
	\cX & \doteq \cW
	\doteq \Ltwo(\Lambda;\R)\,, \quad \Lambda \doteq (0,2)\,,
	\nn\\
	\op{A}\, x
	& \doteq -(2 + \partial)\, x\,,
	\
	x\in \dom(\op{A})
	\doteq \left\{ x\in\cX \, \left| \, \ba{c}
					x \text{ absolutely continuous}
					\text{on } \Lambda\cup\{0\},  \\
					x(0) = 0, \ \partial x\in\cX
				\ea \right. \right\},
	\nn\\
	\sigma\, x
	& \doteq \ts{\frac{1}{\sqrt{2}}} \, x\,, 
	\quad
	\op{C}\, x \doteq \ts{\frac{1}{3}} \int_\Lambda x(\zeta)\, d\zeta\,,
	\quad x\in\dom(\sigma) = \dom(\op{C}) = \cX\,.
	\nn
\end{align}
Attention is restricted to initializations $\op{M}\in\Sigma(\cX)$ and $\optilde{M}\in\Sigma_{\op{M}}(\cX)$ that assume an integral representation of the form 
\begin{align}
	\op{M}\, x = (\op{M}\, x)(\cdot)
	& = \int_\Lambda M(\cdot,\zeta)\, x(\zeta)\, d\zeta\,,
	\label{eq:op-integral}
\end{align}
in which $M\in\Ltwo(\Lambda^2;\R)$ denotes a kernel. Under this restriction, respective solutions of the operator differential Riccati equation \er{eq:op-P} enjoy the same integral representation, with time-indexed kernels denoted respectively by $P_t,\widetilde{P}_t\in\Ltwo(\Lambda^2;\R)$, $t\in\R_{\ge 0}$. Consequently, the operator differential Riccati equation \er{eq:op-P} may be equivalently represented via an integro-differential equation of the form 
\begin{align}
	\pdtone{P_t}{t}(\eta,\zeta)
	& = -4 \, P_t(\eta,\zeta) + \pdtone{P_t}{\eta}(\eta,\zeta) + \pdtone{P_t}{\zeta}(\eta,\zeta) + 
	\ts{\frac{1}{2}} \hspace{-1mm} \int_\Lambda \hspace{-1mm} P_t(\eta,\rho) \, P_t(\rho,\zeta)\, d\rho + \ts{\frac{1}{3}}
	\label{eq:integro-DE}
\end{align}
subject to the boundary and initial conditions
\begin{align}
	P_t(0,\zeta) = 0 = P_t(\eta,0)\,, \quad 
	P_0(\eta,\zeta) = M(\eta,\zeta)\,,
	\label{eq:BC-IC}
\end{align}
for all $t\in\R_{\ge 0}$, $\eta,\zeta\in\Lambda$. Equations \er{eq:integro-DE} and \er{eq:BC-IC} are solved via a textbook application of Runge-Kutta (RK45) on a fine grid to provide a benchmark for application of the recipe of Section \ref{sec:recipe} via representation \er{eq:op-integral}. 
(Note that as the specifics of numerical methods are not the main focus here, only a standard numerical method is employed. More advanced numerical methods may be found in, for instance, \cite{ABK:01,BW:89,BM:12,GR:88,GB:06,R:91}.) Approximation errors generated by the dual-space propagation of Theorem \ref{thm:propagate-fund} relative to the RK45 solution are illustrated in Figure \ref{fig:errors}. The computational advantage illustrated there is due to the application of the time-step doubling iteration alluded to in Section \ref{sec:step-2} in computing $\opBmp{t}$. For reasons of brevity, further details are omitted.

\begin{figure}[h]
\begin{center}
\epsfig{file=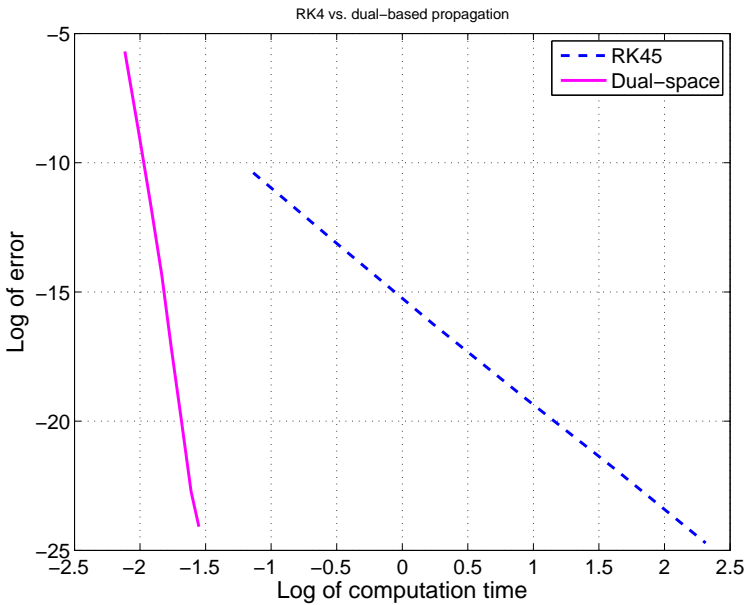,width=10.0cm}
\vspace{-2mm}
\caption{Approximation error versus computation time for standard (RK45) and dual-space propagation methods.}
\label{fig:errors}
\end{center}
\end{figure}


\section{Conclusion}
\label{sec:conc}
By exploiting a connection between an operator differential Riccati equation and a specific infinite dimensional optimal control problem, dynamic programming is employed to develop an evolution operator for propagating solutions of this equation. Examination of this evolution in a dual space, defined via semiconvexity and the Legendre-Fenchel transform, reveals the existence of a time indexed dual space operator that can be used to propagate the solution of the operator differential Riccati equation from any initial condition in a particular class. By demonstrating that these time indexed dual space operators inherit a semigroup property from dynamic programming, the set of such time indexed operators is shown to define a fundamental solution semigroup for the operator differential Riccati equation. 



\bibliographystyle{siam}
\bibliography{operator-Riccati}

\begin{thebibliography}{10}

\bibitem{ABK:01}
{\sc J.~Atwell, J.~Borggaard, and B.~King}, {\em Reduced order controllers for
  {B}urgers' equation with a nonlinear observer}, Int. J. Appl. Math. Comput.
  Sci., 11 (2001), pp.~1311--1330.

\bibitem{BN:00}
{\sc G.~Bachman and L.~Narici}, {\em Functional Analysis}, Dover (orig. Acad.
  Press, 1966), 2000.

\bibitem{BW:89}
{\sc H.~Banks and C.~Wang}, {\em Optimal feedback control of
  infinite-dimensional parabolic evolution systems: Approximation techniques},
  SIAM Journal on Control and Optimization, 27 (1989), pp.~1182--1219.

\bibitem{BCD:97}
{\sc M.~Bardi and I.~Capuzzo-Dolcetta}, {\em Optimal control and viscosity
  solutions of Hamilton-Jacobi-Bellman equations}, Systems \& Control:
  Foundations \& Application, Birkhauser, 1997.

\bibitem{B:57}
{\sc R.~Bellman}, {\em Dynamic programming}, Princeton University Press, 1957.

\bibitem{BM:12}
{\sc P.~Benner and H.~Mena}, {\em Numerical solution of the
  infinite-dimensional {LQR}-problem and the associated differential riccati
  equations}, MPI Magdeburg Preprint MPIMD/12-13 (August),  (2012).

\bibitem{BDDM:07}
{\sc A.~Bensoussan, G.~D. Prato, M.~Delfour, and S.~Mitter}, {\em
  Representation and control of infinite dimensional systems}, Birkha\"{u}ser,
  second~ed., 2007.

\bibitem{CZ:95}
{\sc R.~Curtain and H.~Zwart}, {\em An introduction to infinite-dimensional
  linear systems theory}, vol.~21 of Texts in Applied Mathematics,
  Springer-Verlag, New York, 1995.

\bibitem{DM:11}
{\sc P.~Dower and W.~McEneaney}, {\em A max-plus based fundamental solution for
  a class of infinite dimensional {R}iccati equations}, in Proc. Joint
  $50^{th}$ IEEE Conference on Decision and Control and European Control
  Conference (Orlando FL, USA), 2011, pp.~615--620.

\bibitem{DM:12}
\leavevmode\vrule height 2pt depth -1.6pt width 23pt, {\em A max-plus method
  for the optimal control of a diffusion equation}, in proc. $51^{st}$ IEEE
  Conference on Decision \& Control (Maui HI, USA), 2012, pp.~618--623.

\bibitem{DM:14}
\leavevmode\vrule height 2pt depth -1.6pt width 23pt, {\em A max-plus dual
  space fundamental solution semigroup for operator differential {R}iccati
  equations}, in proc., $21^{st}$ International Symposium on Mathematical
  Theory of Networks and Systems (Groningen, The Netherlands), 2014,
  pp.~148--153.

\bibitem{GR:88}
{\sc J.~Gibson and I.~Rosen}, {\em Numerical approximation for the
  infinite-dimensional discrete-time optimal linear-quadratic regulator
  problem}, SIAM Journal on Control and Optimization, 26 (1988), pp.~428--451.

\bibitem{GB:06}
{\sc S.~G\"orner and P.~Benner}, {\em {MPC} for the {B}urgers equation based on
  an {LQG} design}, Proc. Appl. Math. Mech., 6 (2006), pp.~781---782.

\bibitem{HRS:01}
{\sc R.~Hagen, S.~Roch, and B.~Silbermann}, {\em $\mbox{C}^*$-algebras and
  numerical analysis}, vol.~236 of Pure and Applied Math., Marcel Dekker, 2001.

\bibitem{K:78}
{\sc E.~Kreyszig}, {\em Introductory functional analysis with applications},
  Wiley, 1978.

\bibitem{M:06}
{\sc W.~McEneaney}, {\em Max-plus methods for nonlinear control and
  estimation}, Systems \& Control: Foundations \& Application, Birkhauser,
  2006.

\bibitem{M:08}
\leavevmode\vrule height 2pt depth -1.6pt width 23pt, {\em A new fundamental
  solution for differential {R}iccati equations arising in control},
  Automatica, 44 (2008), pp.~920--936.

\bibitem{P:83}
{\sc A.~Pazy}, {\em Semigroups of linear operators and applications to partial
  differential equations}, vol.~44 of Applied Mathematical Sciences,
  Springer-Verlag, 1983.

\bibitem{R:88}
{\sc W.~Ray}, {\em Real Analysis}, Prentice-Hall, 1988.

\bibitem{R:74}
{\sc R.~Rockafellar}, {\em Conjugate duality and optimization}, SIAM Regional
  Conf. Series in Applied Math., 16 (1974).

\bibitem{R:91}
{\sc I.~Rosen}, {\em Convergence of {G}alerkin approximations for operator
  {R}iccati equations - a nonlinear evolution equation approach}, J. Math.
  Anal. Appl., 155 (1991), pp.~226--248.

\bibitem{TL:80}
{\sc A.~Taylor and D.~Lay}, {\em Introduction to functional analysis}, John
  Wiley, second~ed., 1980.

\end{thebibliography}


\appendix


\section{Continuous and strongly continuous operators}
\label{app:spaces}

Some well-known facts concerning the continuity of operator-valued functions are recalled for completeness, see for example \cite{P:83,CZ:95,BDDM:07} and the references therein.

An operator-valued function $\op{F}:\R\mapsinto\bo(\cX)$ is continuous at $t_0\in\R$ if given $\eps\in\R_{>0}$ there exists an $\delta\in\R_{>0}$ such that
$
	|t-t_0|<\delta
	\Longrightarrow
	\| \op{F}(t) - \op{F}(t_0) \|_{\bo(\cX)} < \eps
$,
in which $\|\op{F}(t)\|_{\bo(\cX)} \doteq \sup_{\|x\|=1} \|\op{F}(t)\, x\|$ denotes the induced operator norm of $\op{F}(t)\in\bo(\cX)$, and $\|\cdot\|$ is the norm on $\cX$. An operator-valued function $\op{F}:\R\mapsinto\bo(\cX)$ is continuous on an interval $I\subset\R$ if it is continuous at every $t_0\in I$. The space of operator-valued functions $C(I;\bo(\cX))$ is defined as the space of all such continuous operator-valued functions defined on $I$.

Similarly, an operator-valued function $\op{F}:\R\mapsinto\bo(\cX)$ is strongly continuous on an interval $I\subset\R$ if, for every $x\in\cX$, the function $\op{F}(\cdot)\, x:I\mapsinto\cX$ is continuous. That is, for any $t_0\in I$ and $x\in\cX$, given $\eps\in\R_{>0}$, there exists a $\delta\in\R_{>0}$ such that
$
	|t - t_0| < \delta
	\Longrightarrow
	\| \op{F}(t)\, x - \op{F}(t_0) \, x\| < \eps
$,
in which $\|\cdot\|$ denotes the norm on $\cX$. The space of strongly continuous operator-value functions $C_0(I;\bo(\cX))$ is defined as the space of all such strongly continuous operator-valued functions defined on $I$. 

These definitions are employed in \er{eq:space-C} and \er{eq:space-C0}.

\begin{lemma}
\label{lem:spaces}
For any compact interval $I\subset\R$,
\begin{align}
	C(I;\bo(\cX)) \subset C_0(I;\bo(\cX)) \equiv \bo(\cX; C(I;\cX))\,,
	\label{eq:ordering}
\end{align}
in which $\| f \|_{C(I;\cX)} \doteq \sup_{t\in I} \| f(t) \|$.
\end{lemma}
\begin{proof}
Fix any $\op{F}\in C(I;\bo(\cX))$, $\eps\in\R_{>0}$, and $x\in\cX$, $\|x\|\ne 0$. Set $\eps_1 \doteq \eps/\|x\|$. As $\op{F}$ is continuous  by definition, there exists a $\delta(\eps_1)\in\R_{>0}$ such that 
$
	|t-t_0| < \delta(\eps_1)
	\Longrightarrow
	\| \op{F}(t) - \op{F}(t_0) \|_{\bo(\cX)} < \eps_1 
$.
However, as $\op{F}(t),\, \op{F}(t_0)\in\bo(\cX)$,
$
	\| \op{F}(t)\, x - \op{F}(t_0)\, x \|
	\le \| \op{F}(t) - \op{F}(t_0) \|_{\op{\cX}} \, \|x\| < \eps_1 \|x\| = \eps
$.
That is,
$
	|t - t_0| \le \delta(\eps / \|x\|)
	\Rightarrow
	\| \op{F}(t)\, x - \op{F}(t_0) \, x\| < \eps
$.
As $\eps\in\R_{>0}$ and $x\in\cX$, $\|x\|\ne 0$, are arbitrary, it follows by definition that $\op{F}$ is strongly continuous. As $\op{F}\in C(I;\bo(\cX))$ is arbitrary, the left-hand relation in \er{eq:ordering} follows immediately.

The right-hand equivalence may be proved by showing that $C_0(I;\bo(\cX))\subseteq\bo(\cX; C(I;\cX))$ and $\bo(\cX; C(I;\cX)) \subseteq C_0(I;\bo(\cX))$. To this end, first fix any $\op{F}\in C_0(I;\bo(\cX))$. As per \cite[p.387]{BDDM:07}, define $f(x)(t) \doteq \op{F}(t)\, x$, and note that $\op{F}$ is strongly continuous. Hence, given any $x\in\cX$, $t_0\in I$, $\eps\in\R_{>0}$, there exists a $\delta\in\R_{>0}$ such that
$
	|t - t_0| < \delta
	\Longrightarrow
	\| f(x)(t) - f(x)(t_0) \| < \eps
$.
That is, $f(x)\in C(I;\cX)$ for each $x\in\cX$. Consequently, for each $x\in\cX$, the function $\|f(x)(\cdot)\|:I\mapsinto\R_{\ge 0}$ must achieve a finite maximum $M_x\in\R_{\ge 0}$ on $I$ by the Extreme Value Theorem. That is, for each $x\in\cX$, $\|f(x)(t)\| = \|\op{F}(t)\, x\| \le M_x < \infty$ for all $t\in I$. Hence, the Uniform Boundedness Theorem (e.g. \cite[Theorem 4.7-3, p.249]{K:78}) implies that there exists an $M\in\R_{\ge 0}$ such that $\sup_{t\in I} \|\op{F}(t)\|_{\bo(\cX)} \le M < \infty$. Consequently,
\begin{align}
	\| f(x) \|_{C(I;\cX)}
	& \doteq \sup_{t\in I} \|f(x)(t)\| = \sup_{t\in I} \| \op{F}(t)\, x\|
	\le \sup_{t\in I} \|\op{F}(t)\|_{\bo(\cX)} \, \| x \| \le M \, \|x\|\,.
	\nn
\end{align}
That is $f:\cX\mapsinto C(I;\cX)$ is bounded. Furthermore, as $f:\cX\mapsinto C(I;\cX)$ satisfies by definition $f(x) = \op{F}(\cdot)\, x$ for all $x\in\cX$, it is linear. Hence, $f\in\bo(\cX;C(I;\cX))$.

Conversely, select any $f\in\bo(\cX;C(I;\cX))$. By definition, there exists a $K\in\R_{\ge 0}$ such that for all $x\in\cX$,
$
	\sup_{t\in I} \|f(x)(t)\| = \|f(x)\|_{C(I;\cX)}
	\le K\, \|x\|
$.
Define $\op{F}(t)\, x \doteq f(x)(t)$, $t\in I$. Hence, $\|\op{F}(t) \, x\| \le K\, \|x\|$ for all $x\in\cX$, so that
\begin{align}
	\op{F}(t) & \in\bo(\cX)
	\quad\quad \forall\  t\in I\,.
	\label{eq:converse-bo}
\end{align}
Furthermore, as $f(x)\in C(I;\cX)$ for every $x\in\cX$, given any $x\in\cX$, $t_0\in I$ and $\eps\in\R_{>0}$, there exists a $\delta\in\R_{>0}$ such that 
\begin{align}
	|t - t_0| < \delta
	\quad\Longrightarrow\quad & |f(x)(t) - f(x)(t_0)| < \eps
	\quad\Longleftrightarrow\quad |\op{F}(t)\, x - \op{F}(t_0) \, x| < \eps\,.
	\nn
\end{align}
That is, $\op{F}:I\mapsinto\bo(\cX)$ is strongly continuous. Recalling \er{eq:space-C0} and \er{eq:converse-bo}, it follows that $\op{F}\in C_0(I;\bo(\cX))$.
\end{proof}

In view of Lemma \ref{lem:spaces}, spaces $C(I;\bo(\cX))$ and $C_0(I;\bo(\cX))$ of \er{eq:space-C} and \er{eq:space-C0} may be equipped with the respective norms
\begin{align}
	\begin{aligned}
	& \hspace{-2mm} \| \op{F} \|_{C\{I\}}
	\doteq \sup_{t\in I} \|\op{F}(t)\|_{\bo(\cX)}\,,
	&& \hspace{-2mm} \op{F}\in C(I;\bo(\cX))\,,
	\\
	& \hspace{-2mm} \| \op{F} \|_{C_0\{I\}}
	\doteq \sup_{\|x\|=1} \|\op{F}(\cdot)\, x\|_{C(I;\cX)}\,,
	&& \hspace{-2mm} \op{F}\in C_0(I;\bo(\cX))\,.
	\end{aligned}
	\label{eq:norm-C-C0}
\end{align}

\begin{lemma}
\label{lem:norms}
$\|\cdot\|_{C\{I\}}$ may be extended to $C_0(I;\bo(\cX))$, whereupon it is equivalent to $\|\cdot\|_{C_0\{I\}}$.
\end{lemma}
\begin{proof}
Fix $\op{F}\in C_0(I;\bo(\cX))$. By definition \er{eq:norm-C-C0},
\begin{align}
	& \|\op{F}\|_{C_0\{I\}}
	= \sup_{\|x\|=1} \| \op{F}(\cdot)\, x\|_{C(I;\cX)}
	= \sup_{\|x\|=1} \sup_{t\in I} \| \op{F}(t)\, x \|
	\nn\\
	& = \sup_{t\in I} \sup_{\|x\|=1} \| \op{F}(t)\, x \|
	= \sup_{t\in I} \| \op{F}(t) \|_{\bo(\cX)}
	\equiv \| \op{F} \|_{C\{I\}}\,.
	\nn
\end{align}

\vspace{-3mm}
\end{proof}

%
%

\begin{lemma}
\label{lem:Banach}
Given any closed interval $I\subset\R$, the normed spaces 
$$
	( C(I;\bo(\cX)),\, \|\cdot\|_{C\{I\}} )\,,
	\qquad
	( C_0(I;\bo(\cX)),\, \|\cdot\|_{C_0\{I\}} )\,,
$$
defined via \er{eq:space-C}, \er{eq:space-C0}, and \er{eq:norm-C-C0}, are Banach spaces.
\end{lemma}
\begin{proof}
The proof that $( C(I;\bo(\cX)),\, \|\cdot\|_{C\{I\}} )$ is a Banach space follows a standard argument (for example, see the proof of \cite[Theorem 4.3.2, p.115]{R:88}) generalized to this setting. In particular, let $\{\op{F}_n\}$ denote a Cauchy sequence in $C(I;\bo(\cX))$. By inspection of \er{eq:norm-C-C0},
$ 
	\| \op{F}_n(t) - \op{F}_m(t)\|_{\bo(\cX)} 
	\le \| \op{F}_n - \op{F}_m\|_{C\{I\}}
$
for any $t\in I$, and $n,m\in\N$. Fixing any $t\in I$, $\{\op{F}_n(t)\}$ defines a Cauchy sequence in $(\bo(\cX), \, \|\cdot\|_{\bo(\cX)})$, which is a Banach space (see for example \cite[Theorem 2.10-1, p.118]{K:78}). Hence, there exists a $\op{F}(t)\in\bo(\cX)$ such that $\op{F}(t) = \lim_{n\rightarrow\infty} \op{F}_n(t)$. Let $\{ t_n \}_{n\in\N}$ denote a sequence in $I$ such that $\lim_{n\rightarrow\infty} t_n = t$. Fix $\eps\in\R_{>0}$ and $N\in\N$ sufficiently large such that $\| \op{F}_N - \op{F} \|_{C\{I\}} < \ts{\frac{\eps}{3}}$. Hence, applying \er{eq:norm-C-C0},  
\begin{align}
	\| \op{F}_N(t_n) - \op{F}(t_n) \| 
	& \le \ts{\frac{\eps}{3}}
	\qquad \forall \ n\in\N.
	\label{eq:Banach-1}
\end{align}
As $\op{F}_N\in C(I;\bo(\cX))$, there exists a $\delta\in\R_{>0}$ such that
\begin{align}
	\left. 
	\ba{r}
	|s-t|<\delta 
	\\
	\ s\in I
	\ea \right\}
	& \ \Longrightarrow \
	\|\op{F}_N(s) - \op{F}_N(t) \|_{\bo(\cX)} < \ts{\frac{\eps}{3}}\,.
	\label{eq:Banach-2}
\end{align}
Furthermore, by definition of $\{t_n\}_{n\in\N}$, there exists an $M\in\N$ sufficiently large such that 
\begin{align}
	n \ge M
	& \quad\Longrightarrow\quad
	|t_n - t| < \delta\,.
	\label{eq:Banach-3}
\end{align}
Hence, for all $n\ge M$, the triangle inequality combined with \er{eq:Banach-1}, \er{eq:Banach-2}. and \er{eq:Banach-3} implies that
\begin{align}
	& \| \op{F}(t_n) - \op{F}(t) \|_{\bo(\cX)}
	\le \| \op{F}(t_n) - \op{F}_N(t_n) \|_{\bo(\cX)} 
	\nn\\
	& \quad + \| \op{F}_N(t_n) - \op{F}_N(t) \|_{\bo(\cX)}
	+ \| \op{F}_N(t) - \op{F}(t) \|_{\bo(\cX)}
	< \eps\,.
	\nn
\end{align}
Hence, as $\eps\in\R{>0}$ is arbitrary, $\op{F}$ must be continuous. That is, $\op{F}\in C(I;\bo(\cX))$, which implies that $( C(I;\bo(\cX)),\, \|\cdot\|_{C\{I\}} )$ is complete, and hence is a Banach space.

In order to prove that $C_0(I;\bo(\cX))$ is a Banach space, recall by the right-hand equivalence of Lemma \ref{lem:spaces} that
$C_0(I;\bo(\cX)) \equiv \bo(\cX; \bo(I;\cX))$, where $\bo(I;\cX)$ is equipped with the norm $\|f\|_{C(I;\cX)} \doteq \sup_{t\in I} \|f(t)\|$. Using the same argument as above, it may be shown that $\cY \doteq ( C(I;\cX), \, \|\cdot\|_{C(I;\cX)} )$ is a Banach space. Hence, applying \cite[Theorem 2.10-2, p.118]{K:78}, $\bo(\cX; \cY)$ is also a Banach space.
\end{proof}


\section{Yosida approximation}
\label{app:Yosida}
The {\em Yosida approximations} \cite{P:83,BDDM:07} of an unbounded densely defined operator $\op{A}:\dom(\op{A})\subset\cX\mapsinto\cX$ refer to a sequence $\{\op{A}_n\}$, $n\in\N$, of approximating bounded linear operators $\op{A}_n\in\bo(\cX)$ that are strongly convergent to $\op{A}$. The following result is classical, see for example \cite{P:83,BDDM:07}, and is provided here for completeness.

\begin{theorem}[Yosida approximation]
\label{thm:Yosida}
Given an unbounded and densely defined operator $\op{A}$ satisfying Assumption \ref{ass:op-A}, there exists a sequence of operators $\{\op{A}_n\}_{n\in\N}\subset\bo(\cX)$, such that the following properties hold:
\begin{enumerate}[(i)]
\item the sequence $\{\op{A}_n\}_{n\in\N}$ is strongly convergent to $\op{A}$ on $\cX$, i.e.
\begin{align}
	\lim_{n\rightarrow\infty} \op{A}_n \, x  = \op{A}\, x\,,
	\qquad \forall \ x\in\dom(\op{A}); \text{ and}
	\label{eq:Yosida-limit}
\end{align}
\item there exists constants $M\in\R_{\ge 1}$ and $\omega\in\R_{\ge 0}$ such that the (respectively, strongly and uniformly continuous) semigroups generated by $\op{A}$ and $\op{A}_n$ satisfy the bound
\begin{align}
	\max\left( \left\|e^{\op{A}\, t}  \right\|_{\bo(\cX)},\, \left\| e^{\op{A}_n\, t} \right\|_{\bo(\cX)} \right)
	& \le M\, e^{\omega\, t}
	\qquad
	\forall \ t\in\R_{\ge 0},\, n\in\N\,.
	\label{eq:Yosida-bounds}
\end{align}
\end{enumerate}
\end{theorem}
The proof of Theorem \ref{thm:Yosida} is standard, see for example \cite{P:83}.
%
%
It may be constructed via a sequence of lemmas. Its foundation is the {\em resolvent} $\op{R}_{\op{A}}:\dom(\op{R}_\op{A})\subset\C\mapsinto\bo(\cX)$ of operator $\op{A}$, which is defined by
\begin{align}
	\begin{aligned}
	\op{R}_{\op{A}}(\lambda)
	& \doteq (\lambda\, \op{I} - \op{A})^{-1}\,,
	\qquad
	\lambda\in\dom(\op{R}_\op{A})\,,
	\\
	\dom(\op{R}_{\op{A}})
	& \doteq \left\{ \lambda\in\C \, \left| \, 
					\ba{c}
					\text{$\lambda\, \op{I} - \op{A}$ is invertible, with}
					\\
					(\lambda\, \op{I} - \op{A})^{-1}\in\bo(\cX)
					\ea \right.
				 \right\}.
	\end{aligned}
	\label{eq:resolvent}
\end{align}
(Note that $\rho(\op{A}) \doteq \dom(\op{R}_\op{A})\subset\C$ is referred to as the {\em resolvent set} of $\op{A}$, while its complement $\sigma(\op{A})\doteq \C\setminus\rho(\op{A})$ is referred to as the {\em spectrum} of $\op{A}$.) 
\begin{lemma}
\label{lem:op-R-properties}
Given an unbounded and densely defined linear operator $\op{A}$ satisfying Assumption \ref{ass:op-A}, there exists  $M\in\R_{\ge 1}$, $\omega_0\in\R_{\ge 0}$, and $n_{\omega_0}\doteq\lceil \omega_0 \rceil+1\in\N$ such that
\begin{align}
	n \in\N_{\ge n_{\omega_0}}
	& \quad\Longrightarrow\quad
	\left\{ \begin{aligned}
		& \| \op{R}_\op{A}(n)\|_{\bo(\cX)} \le \frac{M}{n - \omega_0}\,,
		\\
		& \op{R}_\op{A}(n)\, x \in \dom(\op{A}) \quad \forall \ x\in\cX\,.
	\end{aligned} \right.
	\label{eq:op-R-properties}
\end{align}
\end{lemma}
\begin{proof}
By Assumption \ref{ass:op-A}, $\op{A}$ is the generator of a $C_0$-semigroup of bounded linear operators, with elements $e^{\op{A}\, t}\in\bo(\cX)$ indexed by $t\in\R_{\ge 0}$. Hence, applying \cite[Theorem 2.2, p.4]{P:83}, there exists an $M\in\R_{\ge 1}$, $\omega_0\in\R_{\ge 0}$, such that $\|e^{\op{A}\, t}\|_{\bo(\cX)}\le M\, e^{\omega_0\, t}$ for all $t\in\R_{\ge 0}$. Define $n_{\omega_0} \doteq \lceil \omega_0 \rceil + 1$ as per the lemma statement. Fix any $n\in\N_{\ge n_{\omega_0}}$. The Hille-Yosida Theorem (for example, \cite[Theorem 5.3, p.20]{P:83}) implies that 
\begin{align}
	\| \op{R}_\op{A}(n)\|_{\bo(\cX)} \le \frac{M}{n - \omega_0}\,,
	\label{eq:op-R-bound}
\end{align}
which is as per first assertion of \er{eq:op-R-properties}. Fix any $x\in\cX$. Inequality \er{eq:op-R-bound} immediately implies that $n\, \op{R}_\op{A}(n) - \op{I}\in\bo(\cX)$, so that $\xi\doteq (n\, \op{R}_\op{A}(n) - \op{I})\, x\in\cX$. However,
\begin{align}
	\xi 
	& = (n\, \op{R}_\op{A}(n) - \op{I})\, x
	= \left[ n\, \op{R}_\op{A}(n) - (n\, \op{I} - \op{A}) \, \op{R}_\op{A}(n) \right] x
	= \op{A}\, \op{R}_\op{A}(n)\, x\,.
	\nn
\end{align}
Consequently, as $\xi\in\cX$, it follows immediately that $\op{R}_\op{A}(n)\, x\in\dom(\op{A})$. As $n\in\N_{\ge n_{\omega_0}}$ and $x\in\cX$ are both arbitrary, the proof is complete.
\end{proof}

\begin{lemma}
\label{lem:Yosida-hat}
Given an unbounded and densely defined linear operator $\op{A}$ satisfying Assumption \ref{ass:op-A}, and $n_{\omega_0}\in\N$ as per Lemma \ref{lem:op-R-properties}, the operator 
\begin{align}
	\op{A}_n^\dagger
	& \doteq n\, \op{A}\, \op{R}_\op{A} (n)\,,
	&&
	\dom(\op{A}_n^\dagger) \doteq \cX\,,
	\label{eq:Yosida-hat}
\end{align}
is well-defined for all $n\in\N_{\ge n_{\omega_0}}$ and satisfies the following properties:
\begin{enumerate}[(i)]
\item $\op{A}_n^\dagger = n \, ( n\, \op{R}_\op{A}(n) - \op{I} ) \in\bo(\cX)$ for all $n\in\N_{\ge n_{\omega_0}}$; and
\item $\lim_{n\rightarrow\infty} \op{A}_n^\dagger\, x = \op{A}\, x$ on $\cX$ for all $x\in\cX$.
\end{enumerate}
\end{lemma}
\begin{proof} (The following proof is standard, see for example \cite{BDDM:07}.) Define $n_{\omega_0}\in\N$ as per Lemma \ref{lem:op-R-properties}, and fix any $n\in\N_{\ge n_{\omega_0}}$, $x\in\cX$. Applying Lemma \ref{lem:op-R-properties}, it is immediate from \er{eq:op-R-properties} that $n\, \op{R}_\op{A}(n)\in\bo(\cX)$. Similarly, it is also immediate from \er{eq:op-R-properties} that $\op{R}_\op{A}(n)\, x\in\dom(\op{A})$. Hence, the operator composition $\op{A}\, \op{R}_\op{A}(n)$ in \er{eq:Yosida-hat} is well-posed on $\dom(\op{R}_\op{A}(n)) = \cX$, which implies that $\op{A}_n^\dagger$ as per \er{eq:Yosida-hat} is well-defined with $\dom(\op{A}_n^\dagger) = \cX$. Furthermore, by definition \er{eq:resolvent} of $\op{R}_\op{A}(n)$,
\begin{align}
	\op{A}_n^\dagger\, x
	& = n\, \op{A}\, \op{R}_\op{A}(n)\, x = n\, \op{A}\, (n\,\op{I} - \op{A})^{-1}\, x
	\nn\\
	& = n^2\, (n\,\op{I} - \op{A})^{-1}\, x
	-n\, (n\, \op{I} - \op{A})\, (n\, \op{I} - \op{A})^{-1}\, x
	\nn\\
	& = n^2\, \op{R}_\op{A}(n)\, x - n\, \op{I}\, x
	= n \, ( n\, \op{R}_\op{A}(n) - \op{I} ) \, x\,.
	\label{eq:op-A-n-1}
\end{align}
As $\op{R}_\op{A}(n)\in\bo(\cX)$ by Lemma \ref{lem:op-R-properties},
\begin{align}
	\|\op{A}_n^\dagger\|_{\bo(\cX)}
	& \le n \, ( n\, \|\op{R}_\op{A}(n)\|_{\bo(\cX)} + 1) \le n \left( \frac{n\, M}{n-\omega_0} + 1\right) < \infty\,.
	\nn
\end{align}
As $n\in\N_{\ge n_{\omega_0}}$ is arbitrary, it follows immediately that assertion {\em (i)} holds. In order to prove assertion {\em (ii)}, fix any $n\in\N_{\ge n_{\omega_0}}$, $x,\xi\in\dom(\op{A})$, and set $y\doteq \op{A}\, x$. Further manipulation of \er{eq:op-A-n-1} yields
\begin{align}
	\op{A}_n^\dagger\, x
	& = n^2\, \op{R}_\op{A}(n)\, x - n\, \op{I}\, x
	= n\, \op{R}_\op{A}(n) \left[ n\, \op{I} - (n\, \op{I} - \op{A}) \right] x
	= n\, \op{R}_\op{A}(n)\, \op{A}\, x\,,
	\label{eq:op-A-n-2}
\end{align}
while
\begin{align}
	(n\, \op{R}_\op{A}(n) - \op{I})\, \xi
	& = [n\, \op{R}_\op{A}(n) - \op{R}_\op{A}(n) \, (n\, \op{I} - \op{A})]\, \xi
	= \op{R}_\op{A}(n)\, \op{A}\, \xi\,.
	\label{eq:op-A-n-3}
\end{align}
Hence, writing $(n\, \op{R}_\op{A}(n) - \op{I})\, y = (n\, \op{R}_\op{A}(n) - \op{I})\, \xi + (n\, \op{R}_\op{A}(n) - \op{I})\, (y - \xi)$, 
\begin{align}
	\| (n\, \op{R}_\op{A}(n) - \op{I})\, y\|
	& \le \| (n\, \op{R}_\op{A}(n) - \op{I})\, \xi\| + \|(n\, \op{R}_\op{A}(n) - \op{I})\, (y - \xi)\|
	\nn\\
	& = \| \op{R}_\op{A}(n)\, \op{A}\, \xi\| + \|(n\, \op{R}_\op{A}(n) - \op{I})\, (y - \xi)\|
	\nn\\
	& \le \| \op{R}_\op{A}(n)\|_{\bo(\cX)} \, \|\op{A}\, \xi\| + \|n\, \op{R}_\op{A}(n) - \op{I} \|_{\bo(\cX)} \, \|y - \xi\|
	\nn\\
	& \le \| \op{R}_\op{A}(n)\|_{\bo(\cX)} \, \|\op{A}\, \xi\| + \left( n\, \|\op{R}_\op{A}(n)\|_{\bo(\cX)} + 1 \right) \, \| y - \xi \|
	\label{eq:op-A-n-4}
\end{align}
where the inequalities follow respectively by the triangle inequality and Lemma \ref{lem:op-R-properties} (i.e. that $n\,\op{R}_\op{A}(n)\in\bo(\cX)$), while the equality follows by \er{eq:op-A-n-3}. Subtracting $y = \op{A}\, x$ from both sides of \er{eq:op-A-n-2}, taking the norm, and applying \er{eq:op-A-n-4} then yields
\begin{align}	
	\| \op{A}_n^\dagger\, x - \op{A}\, x\|
	& = \| (n\, \op{R}_\op{A} (n) - \op{I} ) \, \op{A}\, x \|
	\nn\\
	& \le  \| \op{R}_\op{A}(n)\|_{\bo(\cX)} \, \|\op{A}\, \xi\| + (n\, \|\op{R}_\op{A}(n)\|_{\bo(\cX)} + 1 ) \, \| \op{A}\, x - \xi \|\,.
	\label{eq:op-A-n-5}
\end{align}
As $n\in\N_{\ge n_{\omega_0}}$ is arbitrary, \er{eq:op-A-n-5} holds in particular for any $n\in\{n_i\}_{i\in\N}$, where  $\{n_i\}_{i\in\N}$ is any sequence in $\N_{\ge n_{\omega_0}}$ satisfying $\lim_{i\rightarrow\infty} n_i=\infty$. Hence,
\begin{align}
	\lim_{n\rightarrow\infty} & \|\op{A}_n^\dagger\, x - \op{A}\, x\|
	\le \left( \lim_{n\rightarrow\infty} \|\op{R}_\op{A}(n)\| \right) \|\op{A}\, \xi\| + 
	\left(\lim_{n\rightarrow\infty} n\, \| \op{R}_\op{A}(n)\|_{\bo(\cX)} + 1 \right) \, \|\op{A}\, x - \xi\|
	\nn\\
	& \le  \left( \lim_{n\rightarrow\infty} \frac{M}{n-\omega_0} \right) \|\op{A}\, \xi\| + 
	\left(\lim_{n\rightarrow\infty} \frac{n\, M}{n - \omega_0} + 1 \right) \, \|\op{A}\, x - \xi\|
	= M\, \|\op{A}\, x - \xi\|\,,
	\nn
\end{align}
where the right-hand side limits follow by Lemma \ref{lem:op-R-properties}. As $\xi\in\dom(\op{A})$ is arbitrary,
\begin{align}
	\lim_{n\rightarrow\infty} \|\op{A}_n^\dagger\, x - \op{A}\, x\|
	& \le M \left[ \inf_{\xi\in\dom(\op{A})} \| \op{A}\, x - \xi\| \right]
	= 0\,,
	\nn
\end{align}
where $\op{A}\, x\in\cX$ and $\dom(\op{A})$ dense in $\cX$ yields the final equality. As $x\in\cX$ is arbitrary, it follows that assertion {\em (ii)} holds.
\end{proof}

\begin{lemma}
\label{lem:bounds-on-semigroups}
Given an unbounded and densely defined linear operator $\op{A}$ satisfying Assumption \ref{ass:op-A}, there exists $M\in\R_{\ge 1}$ and $\omega\in\R_{\ge 0}$ such that the operators $\op{A}_n^\dagger\in\bo(\cX)$ of \er{eq:Yosida-hat} are well-defined for all $n\in\N_{\ge n_\omega}$, $n_{\omega} \doteq \lceil \omega \rceil +1 \in\N$, with the (respectively $C_0$- and uniformly continuous) semigroups $e^{\op{A}\, t}$ and $e^{\op{A}_n^\dagger\, t}$ generated by $\op{A}$ and $\op{A}_n^\dagger$ satisfying the common bound
\begin{align}
	& \max(\|e^{\op{A}\, t}\|,\, \|e^{\op{A}_n^\dagger\, t} \| )
	\le M\, e^{\omega\, t}
	\qquad
	\forall \ t\in\R_{\ge 0},\, n\in\N_{\ge n_{\omega}}\,.
	\label{eq:bounds-on-semigroups}
\end{align}
\end{lemma}
\begin{proof} The proof follows that provided in \cite{BDDM:07}. In particular, by Assumption \ref{ass:op-A}, $\op{A}$ is the generator of a $C_0$-semigroup of bounded linear operators, with elements $e^{\op{A}\, t}\in\bo(\cX)$ indexed by $t\in\R_{\ge 0}$. Hence, applying \cite[Theorem 2.2, p.4]{P:83} (and as per the proof of Lemma \ref{lem:op-R-properties}), there exists $M\in\R_{\ge 1}$ and $\omega_0\in\R_{\ge 0}$ such that
\begin{align}
	\| e^{\op{A}\, t} \|
	& \le M\, e^{\omega_0\, t}
	\qquad
	\forall \ t\in\R_{\ge 0}\,.
	\label{eq:b-on-s-1}
\end{align}
Define $\omega \doteq 2\, \omega_0\in\R_{\ge 0}$ and $n_\omega \doteq \lceil \omega \rceil +1$ as per the lemma statement (see also \cite[p.102]{BDDM:07}). Fix any $t\in\R_{\ge 0}$ and $n\in\N_{\ge n_\omega}$. Note in particular that $n\ge n_\omega \ge \max(\lceil \omega_0 \rceil + 1,\, \lceil 2\, \omega_0 \rceil)$. Hence, as $n\ge \lceil \omega_0 \rceil + 1$, assertion {\em (i)} of Lemma \ref{lem:Yosida-hat} implies that $\op{A}_n^\dagger\in\bo(\cX)$ is well-defined by \er{eq:Yosida-hat}. Subsequently, \cite[Theorem 1.2, p.2]{P:83} implies that $\op{A}_n^\dagger$ generates a uniformly continuous semigroup of bounded linear operators, each of which enjoys an explicit series representation. In particular, for the $t\in\R_{\ge 0}$ fixed above,
\begin{align}
	e^{\op{A}_n^\dagger\, t}
	& = \sum_{k=0}^\infty \frac{t^k}{k!} \, (\op{A}_n^\dagger)^k
	= e^{-n\, t} \sum_{k=0}^\infty \frac{(n^2\, t)^k}{k!} \, (\op{R}_\op{A}(n))^k\,,
	\label{eq:b-on-s-2}
\end{align}
where the second equality follows by assertion {\em (i)} of Lemma \ref{lem:Yosida-hat} (see \cite[equation (2.30), p.102]{BDDM:07}). The Hille-Yosida Theorem \cite[Theorem 5.3, p.20]{P:83} also implies that
\begin{align}
	\|\op{R}_\op{A}(n)^k\|_{\bo(\cX)}
	& \le \frac{M}{(n-\omega_0)^k}
	\label{eq:b-on-s-3}
\end{align}
for all $k\in\N$, where $M$ and $\omega_0$ are as per \er{eq:b-on-s-1}. Hence, returning to \er{eq:b-on-s-2},
\begin{align}
	& \|e^{\op{A}_n^\dagger\, t}\|_{\bo(\cX)}
	= e^{-n\, t} \left\| \sum_{k=0}^\infty \frac{(n^2\, t)^k}{k!} \, (\op{R}_\op{A}(n))^k \right\|_{\bo(\cX)}
	\le e^{-n\, t} \sum_{k=0}^\infty \frac{(n^2\, t)^k}{k!} \, \| \op{R}_\op{A}(n)^k\|_{\bo(\cX)}
	\nn\\
	& \le e^{-n\, t} \sum_{k=0}^\infty \frac{(n^2\, t)^k}{k!} \, \frac{M}{(n-\omega_0)^k}
	= M\, e^{-n\, t}\, e^{n^2\, t / (n-\omega_0)}
	= M\, e^{ n\, \omega_0\, t / (n - \omega_0) }\,.
	\label{eq:b-on-s-4}
\end{align}
Recalling the definition of $n\in\N_{\ge n_\omega}$, note that $n\ge 2\, \omega_0$. As $\alpha:[2\, \omega_0,\infty)\mapsinto\R$ defined by $\alpha(r) \doteq r\, \omega_0 / (r- \omega_0)$ for all $r\ge 2\, \omega_0$ is monotone decreasing,
\begin{align}
	\frac{n\, \omega_0}{n - \omega_0}
	& \le \alpha(2\, \omega_0) = \frac{2\, \omega_0^2}{2\, \omega_0 - \omega_0} = 2\, \omega_0 = \omega\,.
	\label{eq:b-on-s-5}
\end{align}
Consequently, 
\begin{align}
	\|e^{\op{A}\, t}\|_{\bo(\cX)}
	& \le M\, e^{\omega\, t}\,,
	\qquad
	\|e^{\op{A}_n^\dagger\, t}\|_{\bo(\cX)}
	\le M\, e^{\omega\, t}\,,
	\label{eq:b-on-s-6}
\end{align}
where the first inequality follows by inspection of \er{eq:b-on-s-1} and noting that $\omega_0 \le 2\, \omega_0 = \omega$, and the second inequality follows by substituting \er{eq:b-on-s-5} in \er{eq:b-on-s-4}. As $t\in\R_{\ge 0}$ and $n\in\N_{\ge n_\omega}$ are arbitrary, the proof is complete.
\end{proof}

The proof of Theorem \ref{thm:Yosida} is straightforward in view of Lemmas \ref{lem:Yosida-hat} and \ref{lem:bounds-on-semigroups}.

\begin{proof}[Theorem \ref{thm:Yosida}]
With operator $\op{A}$ fixed as per the theorem statement, let $n_{\omega_0}, \, n_\omega\in\N$ be defined respectively by Lemmas \ref{lem:Yosida-hat} and Lemma \ref{lem:bounds-on-semigroups}. Recalling the proof of the latter lemma, note in particular that $n_\omega\ge n_{\omega_0}$. Consequently, Lemma \ref{lem:Yosida-hat} implies that the operator $\op{A}_n^\dagger$ of \er{eq:Yosida-hat} is well-defined for all $n\in\N_{\ge n_\omega}$. It follows immediately that the operator
\begin{align}
	\op{A}_n
	& \doteq \op{A}_{n-n_\omega +1}^\dagger\,,
	&&
	\dom(\op{A}_n) = \cX
	\label{eq:Yosida}
\end{align}
is well defined for every $n\in\N$. Hence, assertion {\em (ii)} of Lemma \ref{lem:Yosida-hat} implies that assertion {\em (i)} of the hypothesis holds. Similarly, Lemma \ref{lem:bounds-on-semigroups} immediately implies that assertion {\em (ii)} of the hypothesis holds, thereby completing the proof.
\end{proof}

\begin{remark}
As indicated in the statement of Theorem \ref{thm:Yosida}, the operator $\op{A}_n$, $n\in\N$, of \er{eq:Yosida} is referred to throughout as a {\em Yosida approximation}. Elsewhere, for example \cite{P:83}, this terminology is reserved for the operator $\op{A}_n^\dagger$, $n\in\N_{\ge n_\omega}$, of \er{eq:Yosida-hat}.
\end{remark}



\section{Proof of Theorem \ref{thm:optilde-P-existence}}
\label{app:optilde-P-proof}
Fix any $\op{M}\in\Sigma(\cX)$ such that Assumption \ref{ass:coercive} holds, any $\optilde{M}\in\Sigma_\op{M}(\cX)$, and any $n\in\N$. Theorem \ref{thm:existence-uniqueness-P} implies that there exists a $\hat\tau_0\in\R_{>0}$ (independent of $n\in\N$), $\op{P}_n, \, \optilde{P}_n\in C([0,\hat\tau_0];\Sigma(\cX))$, and $\op{P}\in C_0([0,\hat\tau_0];\Sigma(\cX))$ such that $\op{P}_n$ and $\optilde{P}_n$ are unique solutions of \er{eq:op-P-n} satisfying $\op{P}_n(0) = \op{M}$ and $\optilde{P}_n(0) = \optilde{M}$ respectively, and $\op{P}$ is the unique mild solution of \er{eq:op-P} satisfying $\op{P}(0) = \op{M}$. Hence, $\optilde{E}_n\in C([0,\hat\tau_0];\Sigma(\cX))$ is well defined by
\begin{align}
	\optilde{E}_n(t)
	\doteq \optilde{P}_n(t) - \op{P}_n(t)
	\label{eq:op-E-n}
\end{align}
for all $t\in[0,\hat\tau_0]$, where it may be noted that $\optilde{E}_n(0) = \optilde{M} - \op{M}$, which is coercive by definition of $\optilde{M}\in\Sigma_\op{M}(\cX)$. Formally differentiating \er{eq:op-E-n} and applying \er{eq:op-P-n} twice (for $\optilde{P}_n$ and $\op{P}_n$) yields the evolution equation
\begin{align}
	\dot{\optilde{E}}_n(t)
	& = \optilde{L}_n(t)' \, \optilde{E}_n(t) + \optilde{E}_n(t)\, \optilde{L}_n(t)\,,
	\label{eq:evolution-E-tilde}
\end{align}
which holds for all $t\in[0,\hat\tau_0]$, with $\optilde{L}_n(t) \doteq \op{A}_n + \demi\, \sigma\, \sigma' \, ( \op{P}_n(t) + \optilde{P}_n(t) )$, with $\op{A}_n$ denoting the Yosida approximation of $\op{A}$. In order to define the notion of solution for \er{eq:evolution-E-tilde}, it is also useful to consider the related operator evolution equation
\begin{align}
	\dot{\optilde{Y}}_n(t)
	& = \optilde{L}_n(t)' \, \optilde{Y}_n(t)\,.
	\label{eq:evolution-Y-tilde}
\end{align}
A solution of \er{eq:evolution-Y-tilde} on a time interval $[0,T]$, $T\in[0,\hat\tau_1]$, is any operator-valued function $\optilde{Y}_n\in C([0,T];\Sigma(\cX))$ satisfying
\begin{align}
	\optilde{Y}_n(t)\, x
	& = e^{\op{A}_n'\, t}\, \optilde{Y}_n(0)\, x + 
			\demi\, \int_0^t e^{\op{A}_n'\, (t-s)} \, (\sigma\, \sigma' \, [ \op{P}_n(s) + \optilde{P}_n(s) ])' \, \optilde{Y}_n(s)\, x \, ds
	\label{eq:mild-Y-tilde}
\end{align}
for all $x\in\cX$, $t\in[0,T]$, where $e^{\op{A}_n'\, \cdot}$ denotes the uniformly continuous semigroup generated by the Yosida approximation $\op{A}_n'\in\bo(\cX)$ of $\op{A}'$.

\begin{claim}
\label{cla:existence-uniqueness-Y}
Given any $\op{Y}_0\in\Sigma(\cX)$, there exists a $\tau_1\in(0,\hat\tau_1]$ such that the operator evolution equation \er{eq:evolution-Y-tilde} exhibit a unique solution $\optilde{Y}_n\in C([0,\tau_1];\Sigma(\cX))$ satisfying $\optilde{Y}_n(0) = \op{Y}_0$ for all $n\in\N$.
\end{claim}
\begin{proof}[Claim \ref{cla:existence-uniqueness-Y}]
The argument largely follows that of the proof of Theorem \ref{thm:existence-uniqueness-P}, with the only significant departure being that it must be shown that $\|\op{P}_n\|_{C[0,\hat\tau_1]}$ (and $\|\optilde{P}_n\|_{C[0,\hat\tau_1]}$) can be bounded above uniformly with respect to $n\in\N$. To this end, recall that $\op{P}_n\in C([0,\hat\tau_1];\Sigma(\cX))\subset\bo(\cX;C([0,\hat\tau_1];\cX))$ by Lemma \ref{lem:spaces}. Furthermore, Theorem \ref{thm:existence-uniqueness-P} and the triangle inequality imply that
\begin{align}
	0 
	& = \lim_{n\rightarrow\infty} \left\| \op{P}_n(\cdot)\, x - \op{P}(\cdot)\, x \right\|_{C([0,t];\cX)}
	\ge \lim_{n\rightarrow\infty} \left| \left\| \op{P}_n(\cdot) \, x \right\|_{C([0,t];\cX)} - \left\| \op{P}(\cdot) \, x \right\|_{C([0,t];\cX)} \right|
	\nn
\end{align}
for all $x\in\cX$. That is, the sequence $\{ p_n \}_{n\in\N}\subset\R_{\ge 0}$, $p_n \doteq \left\| \op{P}_n(\cdot) \, x \right\|_{C([0,t];\cX)}$, is convergent, and hence bounded. In particular, for each $x\in\cX$, there exists an $M_x\in\R_{\ge 0}$ (independent of $n\in\N$) such that $\| \op{P}_n(\cdot) \, x \|_{C([0,t];\cX)} \le M_x < \infty$ for all $n\in\N$. Consequently, as $\op{P}_n(\cdot) = \op{P}_n\in\bo(\cX;C([0,\hat\tau_1];\cX))$ for all $n\in\N$,  the Uniform Boundedness Theorem (for example \cite[Theorem 4.7-3, p.249]{K:78}) implies that there exists an $M\in\R_{\ge 0}$ such that 
$\|\op{P}_n\|_{\bo(\cX;C([0,\hat\tau_1];\cX))} \le M < \infty$ for all $n\in\N$. Hence, Lemma \ref{lem:norms} implies that $\|\op{P}_n\|_{C[0,\hat\tau_1]}$ is indeed uniformly bounded as required, with
\begin{align}
	\sup_{n\in\N} \|\op{P}_n\|_{C[0,\hat\tau_1]}
	& \le M < \infty\,.
	\nn
\end{align}
With this bound in place (along with the corresponding uniform bound for $\|\optilde{P}_n\|_{C[0,\hat\tau_1]}$), the proof proceeds as per Theorem \ref{thm:existence-uniqueness-P}, with the details omitted for brevity.
\end{proof}

Existence of a unique solution $\optilde{Y}_n$ of \er{eq:evolution-Y-tilde} on $[0,\tau_1]$ as per Claim \ref{cla:existence-uniqueness-Y} implies the existence of an evolution operator that propagates $\optilde{Y}_n$ forward in time. In particular, it defines $\optilde{U}_n:\Delta_{\tau_1}\mapsinto\Sigma(\cX)$, $\Delta_{\tau_1} \doteq \left\{ (t,s)\in\R^2 \, \big| \, t\in[0,\tau_1], \ s\in[0,t] \right\}$, via
\begin{align}
	\optilde{U}_n(t,s)\, \optilde{Y}_n(s)\, x
	= \optilde{Y}_n(t)\, x\,,	
	\label{eq:op-U-n}
\end{align}
for all $x\in\cX$, $(t,s)\in\Delta_{\tau_1}$. This evolution operator $\optilde{U}_n$ satisfy a catalog of properties, see \cite[Theorem 5.2, p.128]{P:83} or \cite[Proposition 3.6, p.138]{BDDM:07}. A subset of these is summarized as follows:
\begin{enumerate}[(i)]
\item $\optilde{U}_n\in C(\Delta_{\tau_1};\bo(\cX))$;
\item $\optilde{U}_n(t,t) = \op{I}$, $\optilde{U}_n(t,s) = \optilde{U}_n(t,r)\, \optilde{U}_n(r,s)$ for all $(r,s),(s,t)\in\Delta_{\tau_1}$; and
\item $\optilde{U}_n$ is differentiable, with $\ts{\pdtone{}{s}} \, \optilde{U}_n(t,s) = -\optilde{U}_n(t,s)\, \optilde{L}_n(s)'$ for all $(t,s)\in\Delta_{\tau_1}$. 
\end{enumerate}
(Where $\optilde{L}_n(t) = \optilde{L}_n\in\bo(\cX)$ is independent of $t$, the resulting evolution operator $\optilde{U}_n(t,s) = \optilde{U}_n(t-s)$ defined by \er{eq:op-U-n} simplifies to the element $e^{(\optilde{L}_n)'\, (t-s)}\in\bo(\cX)$ of the uniformly continuous semigroup generated by $\optilde{L}_n'$.)

Evolution operator $\optilde{U}_n$ facilitates the definition of the notion of solution for the evolution equation \er{eq:evolution-E-tilde}. In particular, a  solution of \er{eq:evolution-E-tilde} on a time interval $[0,T]$, $T\in[0,\tau_1]$, is any operator-valued function $\optilde{E}_n\in C([0,T];\Sigma(\cX))$ satisfying
\begin{align}
	\optilde{E}_n(t)\, x
	& = \optilde{U}_n(t,0)\, \optilde{E}_n(0)\, \optilde{U}_n(t,0)'\, x
	\label{eq:mild-E-tilde}
\end{align}
for all $x\in\cX$, $t\in[0,T]$. (See also Appendix \ref{app:integral-forms}.) Recalling the definition \er{eq:op-E-n} of $\optilde{E}_n$, its required initialization $\optilde{E}_n(0) = \optilde{M} - \op{M}$ is coercive by definition of $\optilde{M}\in\Sigma_\op{M}(\cX)$. Furthermore, $\op{P}(t) - \op{M}$ is coercive for all $t\in(0,\tau_1]$ by Assumption \ref{ass:coercive}. That is, there exists  $\eps_0\in\R_{>0}$ and $\eps_1:[0,\tau_1]\mapsinto\R_{\ge 0}$, satisfying $\eps_1(0) = 0$ and $\eps_1(s)>0$ for all $s\in(0,\tau_1]$, such that
\begin{align}
	& \langle x,\, (\optilde{M} - \op{M})\, x \rangle 
	\ge \eps_0 \, \|x\|^2\,,
	\qquad
	\langle x,\, (\op{P}(t) - \op{M})\, x \rangle
	\ge \eps_1(t)\, \|x\|^2
	\label{eq:coercivity-requirements}
\end{align}
for all $x\in\cX$ and $t\in[0,\tau_1]$. Note in particular that $\eps_1$ is independent of $n\in\N$. Hence, recalling \er{eq:op-E-n} and \er{eq:mild-E-tilde},
\begin{align}
	\langle x,\, & (\optilde{P}_n(t) - \op{M})\, x \rangle
	= \langle x,\, \optilde{E}_n(t)\, x \rangle + \langle x,\, (\op{P}_n(t) - \op{P}(t))\, x \rangle +  \langle x,\, (\op{P}(t) - \op{M})\, x \rangle
	\nn\\
	& = \langle \optilde{U}_n(t,0)'\, x,\, (\optilde{M} - \op{M})\, \optilde{U}_n(t,0)'\, x \rangle 
					+ \langle x,\, (\op{P}_n(t) - \op{P}(t))\, x \rangle +  \langle x,\, (\op{P}(t) - \op{M})\, x \rangle
	\nn\\
	& \ge \eps_0\, \| \optilde{U}_n(t,0)'\, x\|^2 - \|x\|\, \sup_{s\in[0,t]} \|\op{P}_n(s)\, x - \op{P}\, x\| + \eps_1(t)\, \|x\|^2
	\nn\\
	& \ge \eps_1(t)\, \|x\|^2 - \|x\|\, \|\op{P}_n(\cdot)\, x - \op{P}\, x\|_{C([0,t];\cX)}\,,
	\nn
\end{align}
where the first inequality follows by \er{eq:coercivity-requirements} and Cauchy-Schwartz, and the second inequality follows by positivity of $\eps_0$. Hence, taking the limit as $n\rightarrow\infty$, the limit relationship \er{eq:op-limit} of Theorem \ref{thm:existence-uniqueness-P} implies that
\begin{align}
	\langle x,\, (\optilde{P}(t) - \op{M})\, x \rangle
	& \ge \eps_1(t) \, \|x\|^2
\end{align}
for all $x\in\cX$, $t\in[0,\tau_1]$, where $\eps_1(0) = 0$ and $\eps_1(t)>0$ for all $t\in(0,\tau_1]$. That is, $\optilde{P}\in C([0,\tau_1];\Sigma(\cX))\cap C((0,\tau_1];\Sigma_\op{M}(\cX))$, as required. $\square$



\section{Integral forms of operator differential equations}
\label{app:integral-forms}
The notion of a mild solution of an operator differential equation is defined with respect to the corresponding operator integral equation, see for example \er{eq:op-P-mild}, \er{eq:op-Q-mild}, \er{eq:op-R-mild}, \er{eq:mild-dynamics}, \er{eq:mild-Y-tilde}, \er{eq:mild-E-tilde}, etc. In each case, the integral equation is derived formally from the operator differential equation. Two example derivations are included here.

{\em (i) Operator differential Riccati equation \er{eq:op-P}.} 
For convenience, let $\Sigma^\op{A}(\cX) \doteq \{ \op{P}\in\Sigma(\cX)\, \bigl| \, \op{P}\, x\in\dom(\op{A}) \ \forall \ x\in\dom(\op{A})\}$. Given some horizon $t\in\R_{>0}$, let $\op{P}\in C_{0}^1([0,t];\Sigma^\op{A}(\cX))$ denote a (strict) solution of the operator differential Riccati equation \er{eq:op-P-mild}, where $C_0^1([0,t];\Sigma^\op{A}(\cX))$ denotes the corresponding space of strongly Frechet differentiable operator-valued functions defined on $[0,t]$. Define $\pi_t:[0,t]\mapsinto\bo(\cX)$ by $\pi_t(s)\, x \doteq e^{\op{A}'\, (t-s)}\, \op{P}(s)\, e^{\op{A}\, (t-s)}\, x$, $s\in[0,t]$, $x\in\cX$, and note that $\pi_t(0)\, x = e^{\op{A}'\, t}\, \op{P}(0) \, e^{\op{A}\, t}$, $\pi_t(t)\, x = \op{P}(t)\, x$, and $\pi(\cdot)\, x$ is {\Frechet} differentiable. In particular, recalling \cite[Theorem 2.4, p.4]{P:83}, the chain rule for {\Frechet} differentiation, and \er{eq:op-P},
\begin{align}
	\ts{\ddtone{}{s}} [\pi_t(s)\, x]
	& = e^{\op{A}'\, (t-s)} \left[ - \op{A}'\, \op{P}(s) + \opdot{P}(s) + \op{P}(s)\, \op{A} \right] e^{\op{A}\, (t-s)}\, x
	\nn\\
	& = e^{\op{A}'\, (t-s)} \left[ \op{P}(s) \, \sigma\, \sigma' \, \op{P}(s) + \op{C} \right] e^{\op{A}\, (t-s)}\, x\,.
	\nn
\end{align}
Hence, integration yields that
\begin{align}
	\op{P}(t)\, x 
	& = \pi_t(t)\, x = \pi_t(0)\, x + \int_0^t \ts{\ddtone{}{s}} [ \pi_t(s)\, x ] \, ds
	\nn\\
	& = e^{\op{A}'\, t}\, \op{P}(0) \, e^{\op{A}\, t}\, x 
			+ \int_0^t e^{\op{A}'\, s} \left[ \op{P}(s) \, \sigma\, \sigma' \, \op{P}(s) + \op{C} \right] e^{\op{A}\, (t-s)}\, x \, ds
	= \gamma(\op{P})(t)\, x\,,
	\nn
\end{align}
which is the integral form \er{eq:op-P-mild} of the operator differential Riccati equation \er{eq:op-P}.

{\em (ii) Operator evolution equation \er{eq:evolution-E-tilde}.} 
Given a time horizon $t\in\R_{>0}$, let $\optilde{E}_n\in C^1([0,t];\Sigma(\cX))$ denote a (strict) solution of the operator evolution equation \er{eq:evolution-E-tilde}, where $C^1([0,t];\Sigma(\cX))$ denotes the corresponding space of uniformly {\Frechet} differentiable operator-valued functions defined on $[0,t]$. Define $\pi_t:[0,t]\mapsinto\bo(\cX)$ by $\pi_t(s) \doteq \optilde{U}_n(t,s)\, \optilde{E}_n(s)\, \optilde{U}_n(t,s)'$ where the evolution operator $\optilde{U}_n$ is as per \er{eq:op-U-n}, and note that $\pi_t(0) = \optilde{U}_n(t,0)\, \optilde{E}_n(0)\, \optilde{U}_n(t,0)'$, $\pi_t(t) = \optilde{E}_n(t)$, and $\pi_t(\cdot)$ is {\Frechet} differentiable. In particular, recalling property {\em (iii)} of $\optilde{U}$ as set out in Appendix \ref{app:optilde-P-proof} (or \cite[Theorem 5.2, p.128]{P:83}), the chain rule for {\Frechet} differentiation, and \er{eq:evolution-E-tilde},
\begin{align}
	\ts{\ddtone{}{s}} [\pi_t(s)]
	& = \optilde{U}_n(t,s) \left[ -\optilde{L}_n(s)'\, \optilde{E}_n(s) + \dot{\optilde{E}}_n(s) 
										- \optilde{E}_n(s)\, \optilde{L}_n(s) \right] \optilde{U}_n(t,s)'
	= 0\,.
	\nn
\end{align}
Hence, integration yields that
\begin{align}
	\optilde{E}_n(t) = \pi_t(t)
	& = \pi_t(0) = \optilde{U}_n(t,0)\, \optilde{E}_n(0) \, \optilde{U}_n(t,0)'\,,
	\nn
\end{align}
which yields the integral form \er{eq:mild-E-tilde} of the operator evolution equation \er{eq:evolution-E-tilde}.


\section{Quadratic functionals}

\begin{lemma}
\label{lem:quadratic}
Given any $\op{F}\in\Sigma(\cX)$, the quadratic functional $f:\cX\mapsinto\R$ defined by $f(x) \doteq \demi\langle x,\, \op{F}\, x \rangle$, $x\in\cX$, satisfies the following properties: (i) $f$ is closed; and (ii) $f$ is convex if and only if $f$ is nonnegative.
\end{lemma}
\begin{proof} 
{\em (i)} Fix any $x\in\cX$, $\delta\in(0,1]$, and any $h\in\cX$ such that $\|h-x\|\le\delta$.
\begin{align}
	\left| f(x+h) - f(x) \right|
	& \le \demi \left| \left\langle h, \, \left( \op{F} + \op{F}' \right) \,x \right\rangle \right| +
	\demi \left| \langle h,\, \op{F} \, h \rangle \right|
	\nn\\
	& \le \demi \|h\| \left( \left\| \left(\op{F} + \op{F}' \right) x \right\| + \| \op{F}\, h\| \right) 
	\le K \, \|h\|
	\nn
\end{align}
for some $K\in\R_{>0}$, where the inequalities follow respectively by the triangle inequality, Cauchy-Schwartz, and boundedness of $\op{F}$. Hence, $f$ is continuous everywhere on $\cX$, and hence closed.

{\em (ii)} Given $\alpha\in[0,1]$ and applying the definition of the quadratic functional $f$, define the functional $\Delta_{\alpha}:\cX\times\cX\mapsinto\R$ by
\begin{align}
	\Delta_{\alpha}(x,\xi)
	& \doteq \alpha f(x) + (1 - \alpha) f(\xi) - f\left(\alpha\, x + (1 - \alpha) \, \xi\right)
	\nn\\
	& = 
	\alpha (1 - \alpha) \left( \langle x,\, \op{F}\, x\rangle + \langle \xi,\, \op{F} \, \xi \rangle - 
	\langle x,\, \op{F} \, \xi \rangle -
	\langle \xi,\, \op{F}\, x \rangle \right)
	\nn\\
	& = 
	\alpha(1-\alpha) \, \langle x - \xi,\, \op{F} \left( x - \xi \right) \rangle 
	= \alpha (1 - \alpha) \, f(x- \xi)\,,
	\nn
\end{align}
where linearity of $\op{F}$ and properties of the inner product have been used.
Supposing that $f$ is nonnegative, $\Delta_{\alpha}(x,\xi) \ge 0$ for all $x,\,\xi\in\cX$. That is, $f$ is convex. Conversely, if $f$ is convex, then it follows by inspection that $f$ must be nonnegative.
\end{proof}

\begin{lemma}
\label{lem:pseudo}
Given any $\op{F}:\cX\mapsinto\cX$, $\xi\in\cX$, suppose that the quadratic functional $f:\cX\mapsinto\R$ defined by $f(x) \doteq \demi \langle x,\, \op{F}\, x \rangle + \langle x,\, \xi \rangle$, $x\in\cX$, satisfies the property that $\sup_{x\in\cX} f(x) < \infty$. Then, the following properties hold:
\begin{enumerate}[(i)]
\item $\op{F}$ is non-positive;
\item the Moore-Penrose inverse $\op{F}^+$ of $\op{F}$ exists; and 
\item there exists an $x^{*}\in\cX$ such that
\begin{align}
	f(x^{*}) & = \sup_{x\in\cX} f(x) = -\demi \langle \xi,\, \op{F}^{+}\, \xi \rangle\,,
	\quad
	\text{where} \quad
	x^{*} = -\op{F}^{+}\, \xi\,.
	\label{eq:pseudo-max}
\end{align} 
\end{enumerate}
\if{false}

\begin{align}
	\sup_{x\in\cX} f(x) < \infty
	& \quad\Longrightarrow\quad 
	\left\{ \begin{aligned}
	& \text{$\op{F}$ is non-positive}\,,
	\\
	& \text{$\op{F}^{+}$ exists}\,,
	\end{aligned}\right.
	\label{eq:pseudo-exists}
\end{align}
where $\op{F}^{+}$ denotes the Moore-Penrose pseudo-inverse of $\op{F}$. Furthermore, 

\fi
\end{lemma}
\begin{proof}
Assume that $\sup_{x\in\cX} f(x) < \infty$. Fix $\xi\in\cX$.

{\em (i)} Suppose that $\op{F}$ is positive. Given $\eps\in\R_{>0}$, there exists an $\bar{x}\in\cX$ such that $\phi_\xi(\bar{x}) > \eps$, where $\phi_\xi:\cX\mapsinto\R$ is defined by
\begin{align}
	\phi_\xi(\bar{x})
	& \doteq \left\{ \ba{cl}
		2 \min\{ \demi \langle \bar x,\, \op{F} \, \bar x \rangle,\, \langle \bar x,\, \xi\rangle \}\,,
		& \xi\ne 0\,,
		\\[1mm]
		\demi \langle \bar x,\, \op{F} \, \bar x \rangle\,,
		& \xi = 0\,.
	\ea\right.
	\nn
\end{align}
With $k\in\R_{>1}$,
$
	f(k\,\bar x)
	= \textstyle{\frac{k^{2}}{2}} \langle \bar x,\, \op{F}\, \bar x \rangle + k \langle \bar x,\, \xi \rangle
	\ge k\, [\demi\, \langle \bar x,\, \op{F}\, \bar x \rangle + \langle \bar x,\, \xi \rangle]
	\ge k \, \phi_\xi(\bar{x})
	> k\, \eps
$.
Hence, $\sup_{x\in\cX} f(x) \ge \sup_{k>1} f(k\, \bar x) \ge \eps\, \sup_{k>1} k = \infty$, which is a contradiction. That is, $\op{F}$ must be non-positive. 

{\em (ii)}, {\em (iii)} As $-\op{F}$ is a non-negative, self-adjoint, bounded linear operator, a square-root operator $\opsqrt{F}$ exists \cite{BN:00} such that $-\op{F} = \opsqrt{F}'\, \opsqrt{F} = \opsqrt{F}\, \opsqrt{F}$, where $\opsqrt{F}$ is also non-negative, self-adjoint, bounded and linear. Hence, $\langle x,\, \op{F}\, x\rangle = - \langle x,\, \opsqrt{F}'\, \opsqrt{F}\, x\rangle = - \langle \opsqrt{F}\, x,\, \opsqrt{F}\, x \rangle$, and
\begin{align}
	f(x)
	& = -\demi \langle \opsqrt{F}\, x,\, \opsqrt{F}\, x \rangle + \langle x,\, \xi\rangle\,.
	\label{eq:quad-sqrt-form}
\end{align}
Let $\cN(\opsqrt{F})$ and $\cR(\opsqrt{F})$ denote the null and range spaces of $\opsqrt{F}$ respectively. As $\opsqrt{F}$ is self-adjoint, $\cNcom(\opsqrt{F}) = \cR(\opsqrt{F})$ and $\cR(\opsqrt{F})$ is closed (c.f. \cite[Theorem 4.10.1]{TL:80}). Furthermore, $\cX = \cN(\opsqrt{F})\, \hat\oplus \, \cNcom(\opsqrt{F})$, where $\hat\oplus$ denotes the direct sum (c.f. \cite[Theorem 2.7.4]{TL:80}). In particular, with $\xi\in\cX$ as per the lemma statement, 
\begin{align}
	\xi
	& = \xi^{N} + \xi^{R}\,,
	\quad\text{where}\quad
	\xi^{N}\in\cN(\opsqrt{F})\,, \ \xi^{R}\in\cR(\opsqrt{F})\,, \
	\langle \xi^{N},\, \xi^{R} \rangle = 0\,.
	\nn
\end{align}
Suppose $\xi^{N}\ne 0$. Recalling \er{eq:quad-sqrt-form},
$
	f(k\, \xi^N)
	= k \langle \xi^{N},\, \xi^{N} + \xi^{R} \rangle 
			- \textstyle{\frac{k^{2}}{2}} \langle \opsqrt{F}\, \xi^{N},\, \opsqrt{F}\, \xi^{N} \rangle
	= k \| \xi^{N} \|^{2}
$
for any $k\in\R$. In particular, $\sup_{x\in\cX} f(x) \ge \| \xi^{N} \|^{2}\sup_{k\in\R_{>0}} k = \infty$, which is a contradiction. That is, $\xi^{N} = 0$, so that $\xi\in\cR(\opsqrt{F})$. Hence, there exists a $y\in\cX$ such that $\xi = \opsqrt{F}\, y$. Again recalling \er{eq:quad-sqrt-form}, completion of squares yields
\begin{align}
	f(x)
	= \langle x,\, \opsqrt{F}\, y \rangle - \demi\langle \opsqrt{F}\, x,\, \opsqrt{F}\, x \rangle 
	& = \demi\langle y,\, y \rangle - \demi \| \opsqrt{F}\, x - y \|^{2}.
	\label{eq:quad-squares}
\end{align}
As $\cR(\opsqrt{F})$ is closed, $\opsqrt{F}$ has a pseudo inverse \cite{HRS:01}, denoted by $\opsqrt{F}^{+}$. Consequently, by inspection of \er{eq:quad-squares}, the supremum over $\cX$ in \er{eq:pseudo-max} is attained at
$x^{*} = \opsqrt{F}^{+}\, y = \opsqrt{F}^{+} \opsqrt{F}^{+}\, \xi$. As $\op{F}$ is self-adjoint, and $-\op{F} = \opsqrt{F}\, \opsqrt{F}$, it follows that the pseudo-inverse of $\op{F}$ also exists and is given by $\op{F}^{+} = -\opsqrt{F}^{+}\, \opsqrt{F}^{+}$. That is, the maximizer may be rewritten as $x^{*} = -\op{F}^{+}\, \xi$, as per \er{eq:pseudo-max}. Substituting in \er{eq:quad-squares} yields
\begin{align}
	f(x^{*})
	& = \demi\langle y,\, y \rangle 
	= \demi\langle \opsqrt{F}^{+}\, \xi,\, \opsqrt{F}^{+}\, \xi\rangle
	= \demi\langle \xi,\, \opsqrt{F}^{+}\, \opsqrt{F}^{+}\, \xi \rangle
	= -\demi \langle \xi,\, \op{F}^{+}\, \xi \rangle\,,
	\nn
\end{align}
as per \er{eq:pseudo-max}.
\end{proof}


\section{Max-plus integral operators}

\begin{lemma}
\label{lem:op-O-finite}
Consider any max-plus integral operator $\op{O}^\oplus$ of the form \er{eq:op-B} with
\begin{align}
	\op{O}^\oplus\, a
	=
	\left( \op{O}^\oplus \, a \right)(\cdot)
	& \doteq 
	\int_{\cX}^{\oplus} O(\cdot,z) \otimes a(z) \, dz\,,
	\label{eq:op-O}
\end{align}
defined with respect to kernel functional $O:\cX\times\cX\mapsinto\R^-$ and functionals $a:\cX\mapsinto\R^-$. Suppose there exists a functional $\widehat a:\cX\mapsinto\R$ such that $\widehat a, (\op{O}^\oplus\, \widehat a):\cX\mapsinto\R^-$ are finite-valued everywhere on $\cX$. 
Then, the kernel functional $O$ is also finite-valued everywhere on $\cX\times\cX$. That is, $O(y,z) < \infty$ for all $y,z\in\cX$.
\end{lemma}
\begin{proof}
Fix any $y,z\in\cX$. From \er{eq:op-O}, $O(y,z) \le (\op{O}^\oplus\, a)(y) - \widehat a(z) < \infty$ by finiteness of $(\op{O}^\oplus\, \widehat a)(y)$ and $\widehat a(z)$.
\end{proof}

\end{document}